\documentclass[12pt]{amsart}
\usepackage{latexsym,fancyhdr,amssymb,color,amsmath,amsthm,graphicx,listings,comment}
\usepackage[section]{placeins} 
\newtheorem{thm}{Theorem} \newtheorem{propo}{Proposition} 
\newtheorem{lemma}{Lemma}  \newtheorem{coro}{Corollary} 
\definecolor{red1}{rgb}{1,0.9,0.9} \definecolor{blue1}{rgb}{0.9,0.9,1} \definecolor{green1}{rgb}{0.9,1,0.9} 
\definecolor{yellow1}{rgb}{1,1,0.9} \definecolor{yellow2}{rgb}{1,1,0.8}

\let\paragraph\subsection 
\newcommand{\TT}{\mathbb{T}}  \newcommand{\QQ}{\mathbb{Q}} \newcommand{\ZZ}{\mathbb{Z}} \newcommand{\RR}{\mathbb{R}} \newcommand{\CC}{\mathbb{C}} \newcommand{\FF}{\mathbb{F}}
\newcommand{\DD}{\mathbb{D}} \newcommand{\PP}{\mathbb{P}}

\title{An Elementary Dyadic Riemann Hypothesis}
\author{Oliver Knill}
\date{January 11, 2018}
\address{Department of Mathematics \\ Harvard University \\ Cambridge, MA, 02138 }
 % Matrices of integers }

\begin{document}
\maketitle

\begin{abstract}
The connection zeta function of a finite abstract simplicial complex $G$
is defined as $\zeta_L(s)=\sum_{x \in G} \lambda_x^{-s}$, where $\lambda_x$ are
the eigenvalues of the connection Laplacian $L$ defined by $L(x,y)=1$ if
$x$ and $y$ intersect and $0$ else. (I) As a consequence of the spectral formula
$\chi(G)=\sum_x (-1)^{{\rm dim}(x)} = p(G)-n(G)$,
where $p(G)$ is the number of positive eigenvalues and $n(G)$ is the number
of negative eigenvalues of $L$, both the Euler characteristic
$\chi(G)=\zeta(0)-2 i \zeta'(0)/\pi$ as well as determinant 
${\rm det}(L)=e^{\zeta'(0)/\pi}$ can be written in terms of $\zeta$. 
(II) As a consequence of the generalized Cauchy-Binet formula for the
coefficients of the characteristic polynomials of a product of matrices we show 
that for every one-dimensional simplicial complex $G$, the functional equation 
$\zeta_{L^2}(s)=\zeta_{L^2}(-s)$ holds, where $\zeta_{L^2}(s)$ 
is the Zeta function of the positive definite squared connection operator $L^2$ of $G$. 
Equivalently, the spectrum $\sigma$ of the integer matrix 
$L^2$ for a $1$-dimensional complex always satisfies the 
symmetry $\sigma = 1/\sigma$ and the characteristic polynomial
of $L^2$ is palindromic. The functional equation extends to products of one-dimensional
complexes. (III) Explicit expressions for the spectrum of circular connection
Laplacian lead to an explicit entire zeta function in the Barycentric limit. 
The situation is simpler than in the Hodge Laplacian $H=D^2$ case \cite{KnillZeta},
where no functional equation was available. In the connection Laplacian case, the limiting zeta 
function is a generalized hypergeometric function which for an integer $s$ is 
given by an elliptic integral over the real elliptic curve $w^2=(1+z)(1-z)(z^2-4z-1)$, 
which has the analytic involutive symmetry $(z,w) \to (1/z,w/z^2)$.
\end{abstract}

\section{Introduction}

\paragraph{}
Zeta functions bind various parts of mathematics. They can be 
defined dynamically, spectrally, geometrically or arithmetically. 
In a dynamic setup, %defined by Artin and Mazur and extended by Ruelle,
one considers prime periodic points of a dynamical system like
closed geodesic loops or automorphisms. In an analytic setup, the zeta function 
is defined from eigenvalues of an operator, usually the Laplacian.
%as in the Minakshisundaram - Pleijel case. 
In arithmetic cases, one considers Dedekind zeta functions 
of an algebraic number field and more generally at L-functions.
%Hasse-Weil Zeta functions are examples where the zeta function is
%obtained as a sum over closed points on a geometric space.

\paragraph{}
The simplest zeta functions are elementary in the sense that they are given by finite sums and 
defined by a self-adjoint integer valued matrix. In the context of simplicial complexes, 
the Bowen-Lanford zeta functions $\zeta_{BL}(s)=1/{\rm det}(1-s A)$ appeared, where $A$ is an 
adjacency matrix. It is a rational function,
and a special case of an Artin-Mazur-Ruelle zeta function, where the system is a subshift of
finite type defined by $A$. From any selfadjoint matrix $A$ one can also define a spectral zeta 
function. In the case of a positive definite matrix with eigenvalues $\lambda_k$, 
this is an entire function $\sum_k \exp(a_k s)$ with real 
$a_k = \log(\lambda_k)$. Both cases are relevant when we look
at the connection Laplacian $L$ of a finite abstract simplicial complex. 
We know there that $\zeta_{BL}(-1)=(-1)^{f(G)}$, where $f(G)$ is
the number of odd dimensional simplices in $G$ \cite{Unimodularity}. 
In the present paper however, we look at the spectral zeta function of a complex $G$.
It determines the Euler characteristic and in certain cases
the cohomology of the geometry. We then focus on mostly on the $1$-dimensional case, where 
we have more symmetry, in particular a functional equation. 

\paragraph{}
The archetype of all zeta functions is the {\bf Riemann zeta function} 
$\sum_{n=1}^{\infty} n^{-s}$. It is
associated to the circle $\TT=\RR/\ZZ$ and its Dirac operator $D=i \partial_x$ with 
spectrum $\ZZ$ as $D e^{i n x} = -n e^{i n x}$. The spectral picture is
normalized by taking $\zeta_{D^2}(s/2)$ and disregarding the zero eigenvalue.
The relation to the arithmetic of rational primes is given by the Euler formula
$\prod_p (1-p^{-s})^{-1}$ so that $\zeta(s)$ is the Dedekind zeta function of the field of 
rational numbers $\QQ$. Riemann looked at the Chebyshev function $\sum_{n \leq x} \Lambda(n)$
where $\Lambda$ is the Mangoldt function and gave the Riemann-Mangoldt formula
$\psi(x) = x - \sum_w \frac{x^w}{w} - \log(2\pi) - \frac{1}{2} \log(1-x^{-2})$,
where the sum is taken over all non-trivial roots $w$ and the remaining expressions are due
to the single pole respectively the trivial zeros $-2,-4,-6, \dots $. 
Pairing complex conjugated non-trivial roots $w_j=a_k+ib_j=|w_j|e^{i \alpha_j},\overline{w}$ gives
$f_j(x) = e^{\log(x) a_j} 2 \cos( \log(x) b_j-\alpha_j)/|a_j+i b_j|$,
functions playing the tunes of the music of the primes. As the Riemann hypothesis
is widely considered an important open problem in mathematics, the topic has been exposited in many places 
\cite{Bombieri92,Conrey,Derbyshire,MusicPrimes,Sabbagh,Rockmore,Watkins2010,VeenCraats,MazurStein}. 

\paragraph{}
When replacing the circle $\TT$ with a finite circular graph $C_n$,
Zeta functions become entire or rational functions. In the case of spectral zeta functions
we have entire functions. Unlike for the Riemann zeta function, no analytic continuation is required. The 
Barycentric limit $C_{2^n}$ is not the compact topological group $\TT$ but the 
profinite compact topological group $\DD_2$ of dyadic integers. 
We will see that for all $1$-dimensional complexes with connection Laplacian $L$, the 
functional equation $\zeta(s)=\zeta(-s)$ holds for $\zeta(s)=\zeta_{L^2}(s/2)$. 
In the Barycentric limit, we then have still an entire zeta function.

\paragraph{}
While for circular graphs $C_n$, the roots of $\zeta(s)$ are symmetric to the imaginary axes, 
in the Barycentric limit, all eigenvalues appear on the imaginary axes. 
We have an explicit limiting function
but currently still lack a reference which confirms that all roots are on the imaginary axes. 
Unlike in the circle case, the distribution of the roots in the dyadic case is very regular. 
This could be related to the fact that the Pontryagin dual of the dyadic group of integers $\DD_2$ 
(which is a subgroup of the dyadic numbers $\RR_2$ is the Pr\"ufer group
$\PP_2=\RR_2/\DD_2$ which is a divisible Abelian group 
(meaning $n \PP_2 = \PP_2$ for every integer $n \geq 1$), 
while the Pontryagin dual $\ZZ$ of the circle group $\TT=\RR/\ZZ$ is not divisible. 

\paragraph{}
The dyadic analogy is to pair up the spaces as follows: $\RR \leftrightarrow \RR_2$,
$\ZZ \leftrightarrow \DD_2$ and $\TT=\RR/\ZZ \leftrightarrow \PP_2=\RR_2/\DD_2$.
The Riemann zeta function belongs to the non-discrete 
compact topological group $\TT$ with non-divisible dual group $\ZZ$ which features
interesting primes while the Dyadic zeta function belongs to the discrete 
compact topological group $\DD_2$ with divisible dual group $\PP_2$ harbors
no interesting primes.

\paragraph{}
While we deal here with a spectral zeta function, also the Bowen-Lanford Zeta function
$\zeta(z) = 1/\det(1-z A)$ of a finite matrix 
plays a role for simplicial complexes. Using $N_n = {\rm tr}(A^n)$, counting periodic paths, one can
write $-\log(\det(1-z A)) = -{\rm tr}(\log(1-z A)) = \sum_{n=1}^{\infty} z^n N_n/n$.
If $A$ is the adjacency matrix of a graph, then $N_n$ is the number of periodic
points of size $n$ of the subshift of finite type, the natural Markov process
defined by the graph. A summation over all possible prime orbits $x$
of length $|x|$ gives $\zeta(z) = \prod_{x \; {\rm prime}}(1-z^{|x|})^{-1}$.
With $p=e^{|x|}$ and $s=-\log(z)$, this is Ihara
type zeta function $\prod_{x \; {\rm prime}} (1-p(x)^{-s})^{-1}$
looks a bit like the Riemann zeta function $\prod_p (1-p^{-s})^{-1}$.

\paragraph{}
A continuum limit for finite zeta functions on simplicial complexes can be obtained by taking 
Barycentric limits $\lim_{n \to \infty} G_n$ of the complex $G=G_0$. 
The Barycentric refinement $G_{n+1}$ of a complex $G_n$ is the Whitney complex
of the graph, where $G_n$ is the set of vertices and two vertices 
are connected, if one is contained in the other. 
It leads to a limiting almost periodic Laplacian on the compact 
topological group $\DD_2$ of dyadic integers. For the connection Laplacian, the limiting 
Laplacian has a ``mass gap". In the dyadic case, the Laplacian remains nice and
invertible even in the continuum limit. 
The Barycentric limits in the higher dimensional case ${\rm dim}(G) \geq 2$ are not well 
understood, but if we take products \cite{StrongRing} of $1$-dimensional Barycentric limits, 
we get higher dimensional models of spaces for which the physics is trivial in the sense that all Green
function values $L^{-1}(x,y)$ are still bounded. Also, the functional equation extends to products of 
$1$-dimensional simplicial complexes. And also in the Barycentric limit,  we have 
$\zeta_{G \times H}(s)=\zeta_G(s) \zeta_G(s)$ because
$(\sum_k \lambda_k^{-s}) (\sum_l \mu_k^{-s}) = \sum_{k,l} (\lambda_k \mu_l)^{-s}$
in the finite dimensional case. 

\paragraph{}
Similarly as with the limit Hodge Laplacian, where the Barycentric zero-locus is
supported on the line ${\rm Re}(s)=1$, we see here that the roots of the limit connection Laplacian are 
supported on the imaginary axes ${\rm Re}(s)=0$.  The functional equation definitely simplifies the analysis. 
Whatever Laplacian we have in mind (like the Hodge Laplacian, the connection Laplacians or
adjacency matrices), the limiting density of states is universal and only dimension dependent
\cite{KnillBarycentric,KnillBarycentric2}. While this universality holds in any dimension, only the 
$1$-dimensional case, it appears to be ``integrable" in the sense that the expressions of the density of states
are explicit. In higher dimensions we see density of states measures which have gaps and possibly
singular parts. Also the roots of the zeta function do not appear to
obey a linear functional equation any more. We will explore in a futher paper that there are truly
higher dimensional cases (not only products of one dimensional cases), 
where the operator $L$ can be deformed to have a functional equation as in one dimension.

\paragraph{}
When looking at finite dimensional approximations of a limiting case, we have to define
convergence. The measure supported by the roots of the zeta function converges in the following
sense: take a compact set $K$ and put a Dirac measure $1$ on each root of the Zeta function to get
a measure $\mu_n$ for each Barycentric refinement $G_n$ of $G$. 
We will show that $\mu_n$ restricted to $K$ 
converges weakly to a limiting measure $\mu$ on $K$. As we know the function $f$ which has $\mu$ as the
zero-locus measure, we expect the measure $\mu$ to sit on the imaginary axes ${\rm Re}(z)=0$. 
As we have shown in the Hodge paper \cite{KnillZeta}, every accumulation point lies there on the 
axes ${\rm Re}(s)=1$. This corresponds to ${\rm Re}(s)=1/2$ for the Dirac case. Some relation of 
the finite dimensional and Riemann zeta case has been pointed out in \cite{FriedliKarlsson}.

\paragraph{}
Spectral zeta function are defined for any geometric object with a Laplacian; in particular
they are defined for manifolds or simplicial complexes.
An inverse spectral problem is to read off the cohomology of the geometry from $\zeta$. 
We have seen that the spectrum of $L$ does not determine the cohomology, even for
$1$-dimensional complexes \cite{HearingEulerCharacteristic}. 
We will see however in future work that the cohomology can easily be read of from the
spectrum of $1$-dimensional complexes which are Barycentric refinements. The connection is very direct: 
the eigenvalues $\lambda=1$ correspond to eigenfunctions supported on a zero-dimensional
connected component of the complex: as a Barycentric refinement of a $1$-dimensional 
complex is bipartite, we can take a $2$-coloring and take a coloring of the vertices 
with with $1$ and $-1$.  This is an eigenfunction as every edge intersects both a value 
$1$ and $-1$ canceling. The eigenvalues $-1$ correspond to eigenfunctions supported on 
the $1$-dimensional parts of the complex. 
A basis for the eigenspace of $\lambda=-1$ can be given by picking 
generators of the fundamental group of the complex and color the edges of that loop 
alternatively with $1$ or $-1$. In some sense, the connection case entangles the zero
and one form eigenvalues nicely. 

\section{Related literature} 

\paragraph{}
The spectral zeta function considered here is 
geometrically defined. It is completely unrelated to a zeta function defined 
in \cite{TaoStructureRandomness} where for any finite
field like $F=\FF_p$, denoting the absolute value of a ``number" $n \in F[t]$
is $|F|^{{\rm deg}(n)}$, the zeta function is $\zeta(s) = \sum_{n \in F[t]_1} |n|^{-s}$
where $F[t]_1$ is the set of monomials in $F[t]$. There is then an explicit
formula $\zeta(s) = 1/(1-p^{s-1})$. 

\paragraph{}
Also completely unrelated is a zeta function for a simplicial
complex $G$ with zero-dimensional components $\{1, \dots, n\}$,
defined in \cite{BjoernerSarkaria}. Given a finite field
$F=F_{q^k}$ one can look at the projective $(n-1)$-space $P$ over $F$ and
let for $x \in G$ the set $V(x,F_{q^k})$ of points in $P$ which
are supported in $x$. As a union of linear subspaces, this is a
projective variety. The zeta function is then
$Z_G(t) = \exp(\sum_{k \geq 1} |V(x,F_{q^k})| t^k/k)$. One of their results
is that $Z_G(t) = \prod_{j=0}^d (1-q^j t)^{-f_j}$, where $d$ is the
dimension of $G$ and $f_j=\chi(G_j)$, where $G_j$ is the $j$'th co-skeleton,
the Whitney complex of the graph with vertex set $\{ x \in G | {\rm dim}(x) \geq j\}$
in which two elements are connected if one is contained in the other. 

\paragraph{}
In \cite{FriedliKarlsson} (apparently unaware of \cite{KnillZeta} which was never
submitted to a journal), the Hodge spectral function 
$\zeta_n(s) = 4^{-s} \sum_{k=1}^n \sin^{2s}(\pi k/n)$
of circular graphs, (which is $\zeta_n(2s)$ in (2) of \cite{KnillZeta}) and
$\zeta_Z(s)=\frac{4^{-s} \Gamma \left(\frac{1}{2}-s\right)}{\sqrt{\pi } \Gamma (1-s)}$
called there  spectral function of the $Z$ are considered. It is
the function $c(2s)$ on page 6 of \cite{KnillZeta}. They then look at
the entire completion
$\xi_Z(s) = 2^s \cos(\pi s/2) \zeta_z(s/2)$ and note
(Theorem 0.2 in \cite{FriedliKarlsson})
that it satisfies the functional equation $\xi_Z(s)=\xi_Z(1-s)$.
Their Theorem 0.3 is close to Theorem 10 and Proposition 9 in \cite{KnillZeta}. Also
some integer values are computed in both papers. 

\paragraph{}
In \cite{FriedliKarlsson} is a nice statement which
shows that $\zeta_n$ has some relation with the Riemann zeta function: introducing
$$  h_n(s)=(4\pi)^{s/2} \Gamma(s/2) n^{-s}(\zeta_n(s/2)-n\zeta_Z(s/2)  \;  $$
the authors there show first that that
$$  \lim_{n \to \infty} |h_n(1-s)/h_n(s)|=1 $$
for every $s$ in the critical strip $0<Re(s)<1$ for which $\zeta(s) \neq 0$. 
Then they show that the statement holds for all $s$ in the critical strip,
if and only the Riemann hypothesis is true. This is hardly a prospective path to prove the Riemann
hypothesis becausse the condition essentially just probes in a fancy way whether
$\zeta(s)$ is zero or not. It is still very interesting as it relates the
Riemann hypothesis with an error in concrete Riemann sums of a concrete 
function parametrized by a complex parameter $s$.

\section{The Morse index}

\paragraph{}
Any selfadjoint matrix $A$ defines a spectral zeta function.
If $\lambda_j$ are the non-zero eigenvalues of an invertible matrix $A$
then its spectral Zeta function is defined as 
$$ \zeta(s) = \sum_k \lambda_k^{-s} \; . $$ 
It satisfies $\zeta'(0) = - \sum_j \log(\lambda_j) 
= -\log( \prod_j \lambda_j) = -\log({\rm det}(A))$.
It is this relation which allows to regularize determinants in geometric settings
where the usual determinant does not make sense. 
This Minakshisundaram - Pleijel approach is used in differential geometry for Laplacians
of a manifold $M$. In all the spectral settings related to $\zeta$ functions, one
disregards the zero eigenvalues, and deal so with 
pseudo determinants. Zeta functions was one motivation to study them better 
in \cite{cauchybinet}. 

\paragraph{}
Like determinants or traces, also the pseudo determinant is one of the
coefficients of the characteristic polynomial $p(x)={\rm det}(A-xI)$. Actually, 
in some sense, the characteristic polynomial is already a zeta function 
as $p(A,z) = (-z)^n/\zeta(1/z)$ for $\zeta(z) = 1/\det(1-z A)$. The roots of the 
characteristic polynomial are the reciprocals of the poles of $\zeta$.
The analogy puts the zeros of a zeta function in the vicinity of the
importance of eigenvalues of a matrix.

% COMMENT 1

\paragraph{}
If $M=\TT$ is the circle then the spectral zeta function is the
classical Riemann zeta function and $\zeta'(0)=-\log(2\pi)/2$. 
The Dirac operator $i d/dx$ has the regularized determinant ${\rm det}(D)=\sqrt{2\pi}$ and 
its square $H=D^2$ satisfies ${\rm det}(H)={\rm det}(D^2)={\rm det}(D)^2 =2\pi$.
The determinant of the Laplacian $H=-d^2/dx^2$ on the circle is the circumference of the circle. 
The zeta regularization shows that one can make sense of the infinite diverging 
product $\prod_{n} n^2$. The classical Riemann example illustrates also how to wash away the 
ambiguities coming from negative eigenvalues by taking $\zeta_{A^2}(s/2)$ rather than $\zeta_A(s)$. 
The example also illustrates the need to transition from determinants to pseudo-determinants.
Riemann however did not look at the zeta function as a spectral 
zeta function but as a tool to investigate primes. 

\paragraph{}
A finite abstract simplicial complex $G$ is a finite set of non-empty sets 
which is closed under the process of taking finite non-empty subsets. 
The connection Laplacian $L$ of a finite abstract simplicial complex $G$ is the 
$n \times n$ matrix $L$ which satisfies $L(x,y)=1$ if two simplices $x,y \in G$ 
intersect and where $L(x,y)=0$ otherwise. For the zeta function $\zeta_L$ of the
connection Laplacian $L$, we have

\begin{propo}
$\zeta_L'(0) = -i \pi n(G)$
\end{propo}
\begin{proof}
The unimodularity theorem $|{\rm det}(L)|=1$ implies ${\rm Re}( \log(L) )= 0$ \cite{Unimodularity}. 
Now ${\rm Im}(\log(L)) = i\pi (n(G))$, where $n(G)$ is the number of negative eigenvalues of $L$.
The fact that $n(G)$ is the number of odd-dimensional simplices was proven in 
\cite{HearingEulerCharacteristic}. 
\end{proof}

\begin{coro}[Euler from Zeta]
$\chi(G) = \zeta(0)-2 i \zeta'(0)/\pi$
\label{eulerfromzeta}
\end{coro}
\begin{proof}
We have $\zeta(0)=p(G)+n(G)$ and $i \zeta'(0)/\pi = n(G)$
so that $\chi(G)=p(G)-n(G) = \zeta(0)-2 i \zeta'(0)/\pi$. 
\end{proof}

% COMMENT 2

\paragraph{}
If we think of $L$ as a Hessian matrix, then the number of negative eigenvalues of $L$
is a Morse index of $L$. We have now, collecting the already established
equality of odd-dimensional simplices and negative eigenvalues: 

\begin{propo}
The Morse index $n(G)$ of $L$ is equal to the number of odd-dimensional simplices in $G$. 
The Poincar\'e-Hopf type index $(-1)^{n(G)}$ satisfies 
$(-1)^{n(G)} = \exp(-i \zeta'_L(0)) = {\rm det}(L(G))$. 
\end{propo} 

\begin{proof}
The fact that $(-1)^{n(G)} = {\rm det}(L(G))$ is the unimodularity theorem. 
\end{proof}

\paragraph{}
There is an obvious ambiguity when defining a zeta function of a general self-adjoint matrix:
if $\lambda$ is a negative number, then $\lambda^{-s} = e^{-s \log(\lambda)}$ requires to chose 
a branch of the logarithm. If $\lambda$ is real, there is no ambiguity. 
But even if we chose a definite branch, the corresponding zeta function is not so nice. 
This is illustrated in Figure~(\ref{2}) where the pictures show the case of a matrix 
with positive and negative spectrum and to the right the zeta function of the square. 
To remove any ambiguities, we can look at 
$$   \zeta(s) = \zeta_{A^2}(s/2) $$ 
as the normalization of the spectral zeta function of a self-adjoint matrix $A$. 
In this paper, we compute in all pictures with $\zeta_{A^2}(s)$, skipping the factor $2$. 

\paragraph{}
The normalization contains less information because we lose information which eigenvalues
are negative and positive. The analytic properties are nicer however for the normalization.
As we have just seen, we can recover the Euler characteristic of a
simplicial complex $G$ from $\zeta_L$ and the trace of $L$ but we have not figured out 
yet whether it is possible to recover $\chi(G)$ from the regularized $\zeta_{L^2}(s/2)$ rather
than from $\zeta(L)$ as in Corollary~(\ref{eulerfromzeta}). The normalization is actually 
something quite familiar: the traditional Riemann zeta function is a normalization 
in this respect already as it ignores the negative eigenvalues of the selfadjoint Dirac operator 
$D=i \partial_x$ on the circle. The Zeta function $\zeta_{D^2}(s/2)$ is then the 
sum $\sum_{n \geq 1} n^{-s}$ of Riemann. 

\begin{figure}
\scalebox{0.12}{\includegraphics{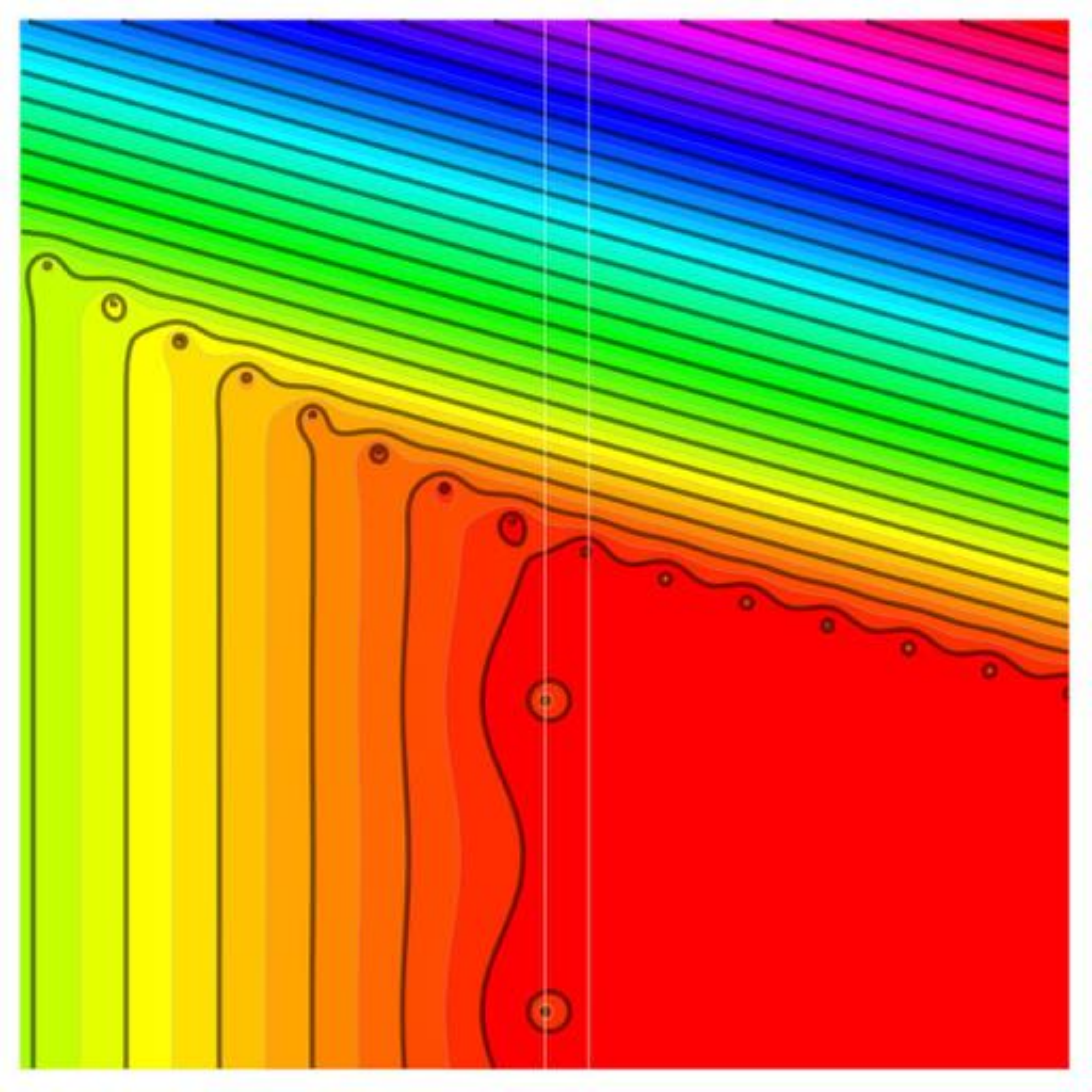}}
\scalebox{0.12}{\includegraphics{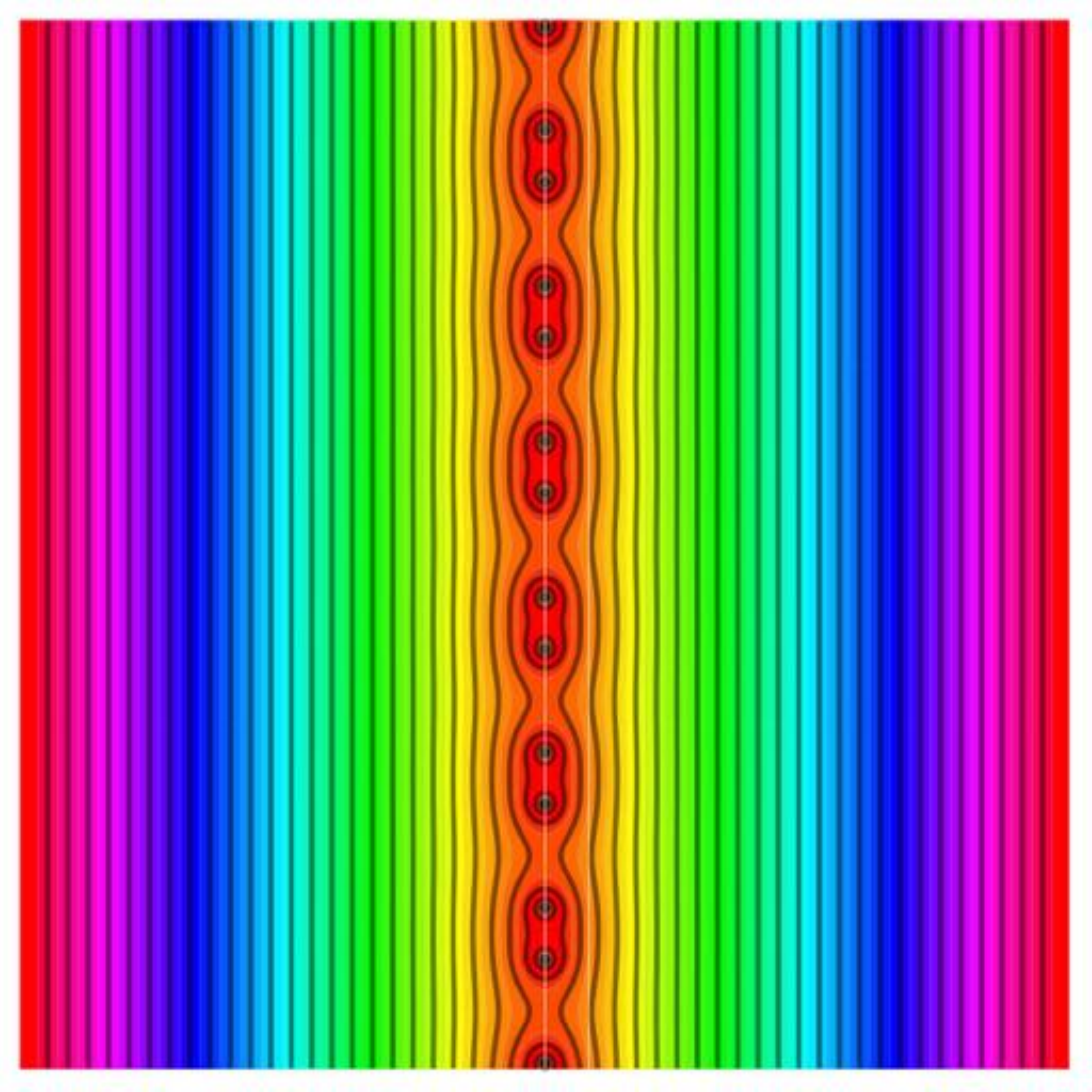}}
\caption{
\label{2}
To the left we see the level curves of $|\zeta_L(s)|=|1+(1-\sqrt{2})^{-s}+(1+\sqrt{2})^{-s}|$
for $G=K_2$. On the right is the normalized case $\zeta_{L^2}(s) = 1+a^{-s}+a^{s}$ with
$a=(1+\sqrt{2})^2$. For simplicity, we don't divide $s$ by $2$.
}
\end{figure}

\paragraph{}
For $G=K_3=\{\{1\},\{2\},\{3\},\{1,2\},\{1,3\},\{2,3\},\{1,2,3\}\}$, the 
eigenvalues of $L$ are 
$\sigma(L)= \{5.511$, $1.618$, $1.618$,$-0.7525$, $-0.618$, $-0.618$, $0.241 \}$
so that 
$$ \zeta_L(s)=(-0.7525)^{-s}+2(-0.618)^{-s}+0.241^{-s}+2 (1.618^{-s})+5.511^{-s} \;  $$
satisfying $\zeta(0)={\rm tr}(I)=7$ and $\zeta'(0)=-3\pi i$.
The normalized $\zeta_{L^2}(s)$ has $\zeta(1)={\rm tr}(L^2)=25$.
But now, $\zeta'(0)=0$ and ${\rm det}(L^2)=1$.
Unlike in the $1$-dimensional case, the spectral symmetry fails,
$\zeta_{L^2}(-1)=37$. Because of unimodularity, the zeta values are integers
for all integer $s$. In this case $\{ \zeta(-3), \dots, \zeta(3) \}$, it is
the list $\{28063, 937, 37, 7, 25, 313, 5131\}$.

\begin{figure}
\scalebox{0.035}{\includegraphics{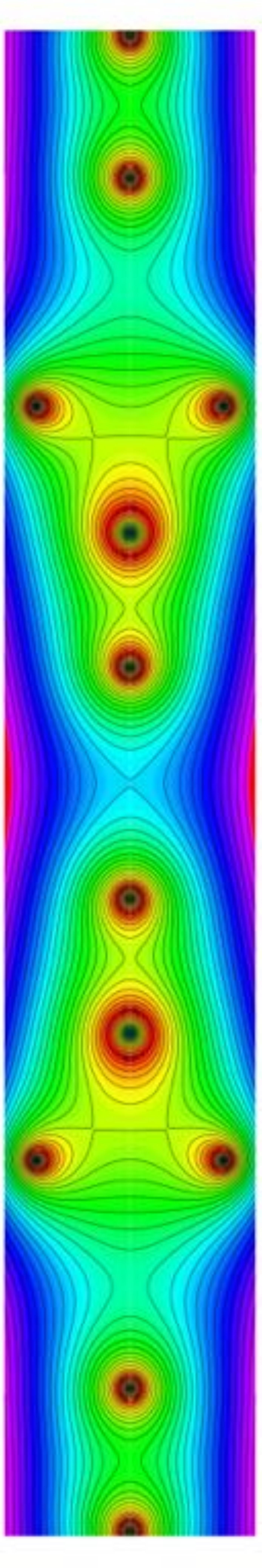}}
\scalebox{0.035}{\includegraphics{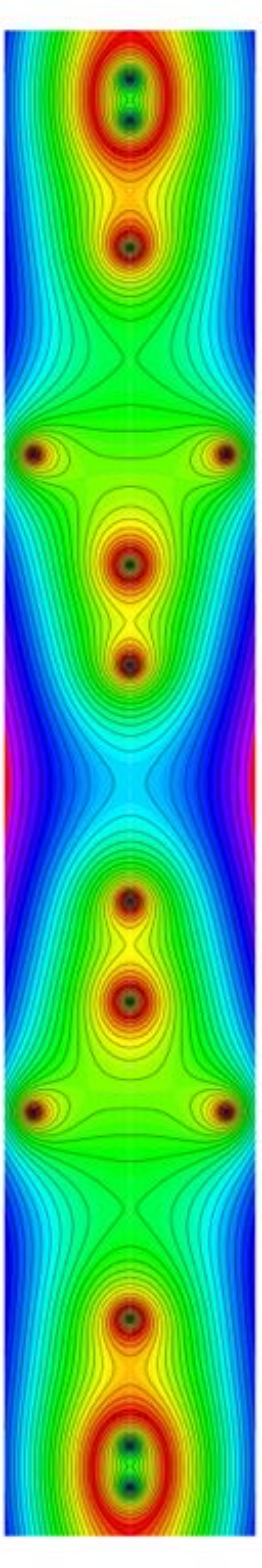}}
\scalebox{0.035}{\includegraphics{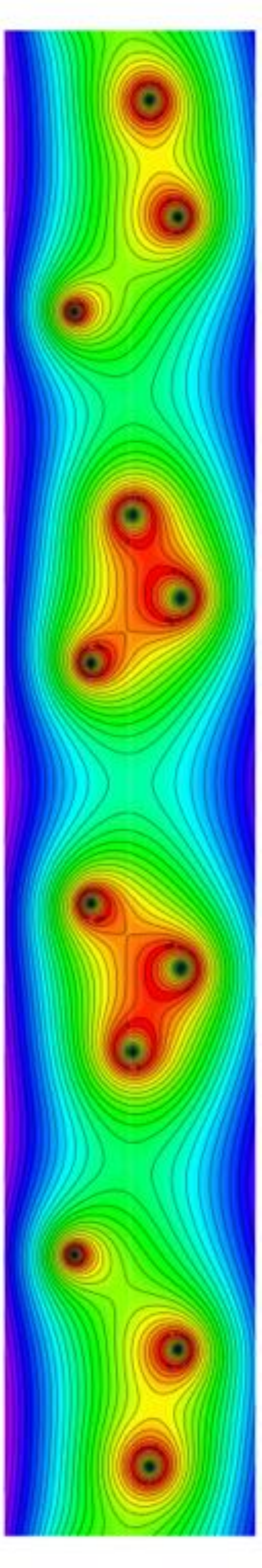}}
\scalebox{0.035}{\includegraphics{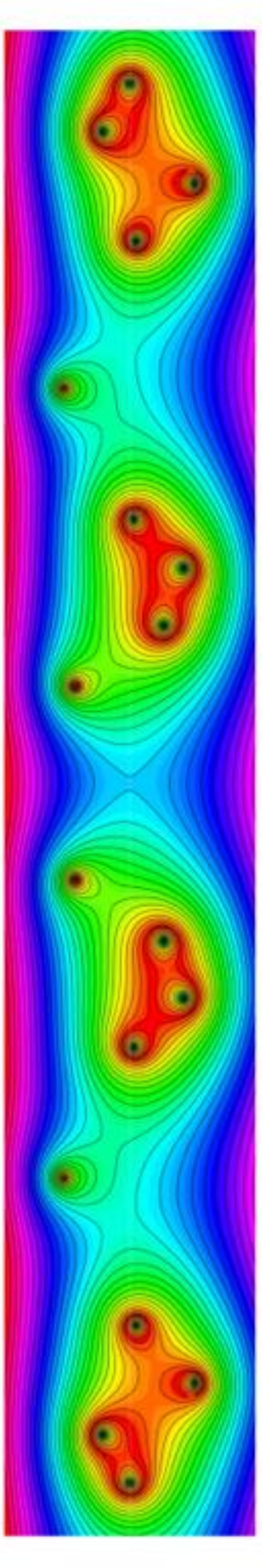}}
\scalebox{0.035}{\includegraphics{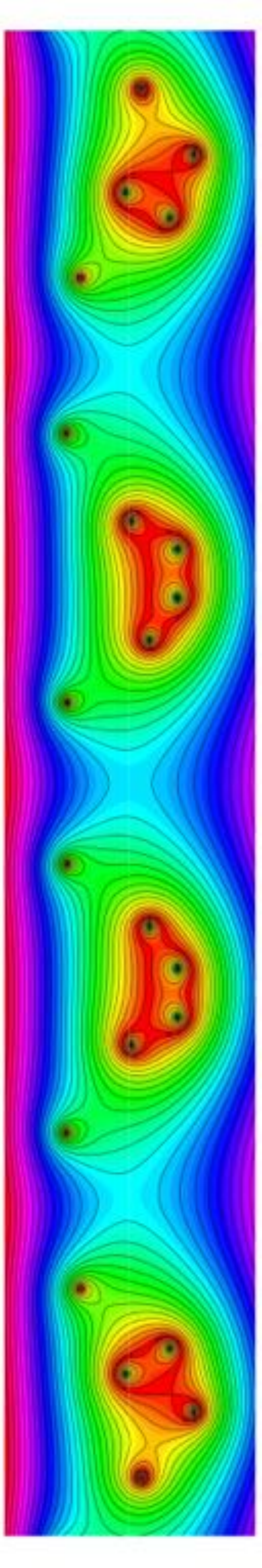}}
\scalebox{0.035}{\includegraphics{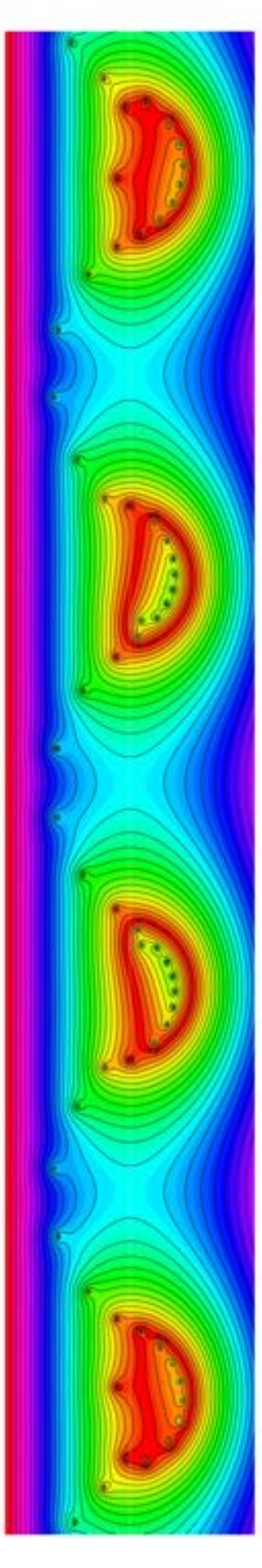}}
\caption{
The level curves of $|\zeta_{L^2}(s)|$ for $C_4$,
figure eight $F_8=C_4 \dot{\cup} C_4$ and
$K_3,K_4,K_5$ and $K_{12}$ on $\{ s=x+iy \; | \; |x|<1,|y|<6 \}$. 
The first two cases show the spectral symmetry $\zeta(s)=\zeta(-s)$.
}
\end{figure}

\section{The roots}

%  COMMENT 3

\paragraph{} 
If $G \times H$ is the product of two simplicial complexes $G,H$ in the strong
ring, then we know that the spectra of $G$ and $H$ multiply. This immediately
implies that the zeta functions multiply. This holds in full generality for
all simplicial complexes $G,H$: 

\begin{coro}
If $\zeta_G(s)$ and $\zeta_H(s)$ are the zeta functions of two simplicial 
complexes $G$ and $H$, then $\zeta_{G \times H}(s) = \zeta_G(s) \zeta_H(s)$. 
\end{coro}

In the subring of the strong ring generated by $1$-dimensional simplices, one has

\begin{coro}
An arbitrary product of $1$-dimensional simplicial complexes satisfies
the functional equation. 
\end{coro}

If we think of an element of the strong ring as a "number", then the spectral 
symmetry applies to the subring generated by $1$-dimensional "primes".
We turn to the symmetry. 

\section{The functional equation} 

\paragraph{} A finite abstract simplicial complex $G$ is $1$-dimensional, 
written ${\rm dim}(G)=1$, if its clique number is $2$. 
This means that all $x \in G$ have cardinality $|x| \leq 2$ but that there
exists at least one simplex $x$ with $|x|=2$. 
Any such complex is the $1$-skeleton complex of a finite simple graph.
(Most graph theory books view graphs as $1$-dimensional simplicial complexes. 
We like to use graphs also in higher dimensional cases as the Whitney complex is a natural 
complex. ) The next theorem therefore can be seen as a result in old-fashioned graph 
theory which looks at graphs as $1$-dimensional simplicial complexes. It will be proven in the 
next section:

\begin{thm}[Functional equation]
For a $1$-dimensional complex, we have $\zeta_{L^2}(-s)=\zeta_{L^2}(s)$. 
\end{thm}

\paragraph{}
It follows from the definition of the zeta function $\zeta(s) = \sum_k \lambda_k^{-s}$ 
that this statement is equivalent to: 

\begin{coro}[Spectral symmetry]
The spectrum of the square $L^2$ satisfies $\sigma(L^2)=\sigma(L^{-2})$. 
\end{coro}

\paragraph{}
Using the sign notation of (\ref{characteristicpolynomial}),
these two statements are again equivalent to 

\begin{coro}[Palindromic characteristic polynomial]
If $p_k$ are the coefficients of the characteristic polynomial of $L^2$ belonging to 
a $1$-dimensional complex, then $p_{n-k} = p_k$.  
\end{coro}

\paragraph{}
The statements can be rephrased also that the 
union of eigenvalues of the multiplicative group generated by $G$ in the 
ring of simplicial complexes produce a finitely generated Abelian group. 
We can look at the rank of this group hoping to get a combinatorial invariant.

\paragraph{}
As a comparison, we know that the Dirac operator $D=d+d^*$ defined by incidence matrices $d$ 
has a spectrum with an additive symmetry $\sigma(D) = -\sigma(D)$. 
But this is hardly a good analogy because the additive symmetry
for $D$ is true for all simplicial complexes, while the functional equation symmetry we 
are currently look at only holds in the $1$-dimensional case. 
Because both the connection matrix $L$ as well as the Dirac operator $D$ have 
both positive and negative spectrum, we like
to compare $L$ and $D$ and see $L^2$ as an analog of the Hodge operator $H=D^2$. 
There are more and more indications which support this point of view like the
exciting formula $H=L-L^{-1}$ (which holds for a suitable basis) we will cover elsewhere. 
We called the operator $L-L^{-1}$ the Hydrogen operator \cite{DehnSommerville}. 

\paragraph{}
As before, we proceed inductively by building up the complex as a CW-complex. First start with a zero-
dimensional complex which is built by cells attached to $-1$-dimensional spheres, the empty complexes.
As the dimension of $G$ to $1$, we now only need to understand what happens if we add a 
$1$-dimensional edge. As mentioned above, the spectral symmetry is equivalent
to a palindromic characteristic polynomial in which we use the notation
\begin{equation}
\label{characteristicpolynomial}
  p(x) = {\rm det}(A-x I) = p_0(-x)^n + \cdots +p_k(-x)^{n-k} + \dots + p_n \; , 
\end{equation}
so that $p_0=1, p_1={\rm tr}(A)$ and $p_n={\rm det}(A)$ and where the largest 
non-zero element $p_k$ is the pseudo determinant of $A$. The validity of the functional 
equation is equivalent to establishing the palindromic symmetry $p_k=p_{n-k}$ for $A=L^2$. 

\begin{figure}
\scalebox{0.04}{\includegraphics{figures/tetrahedron.pdf}}
\scalebox{0.04}{\includegraphics{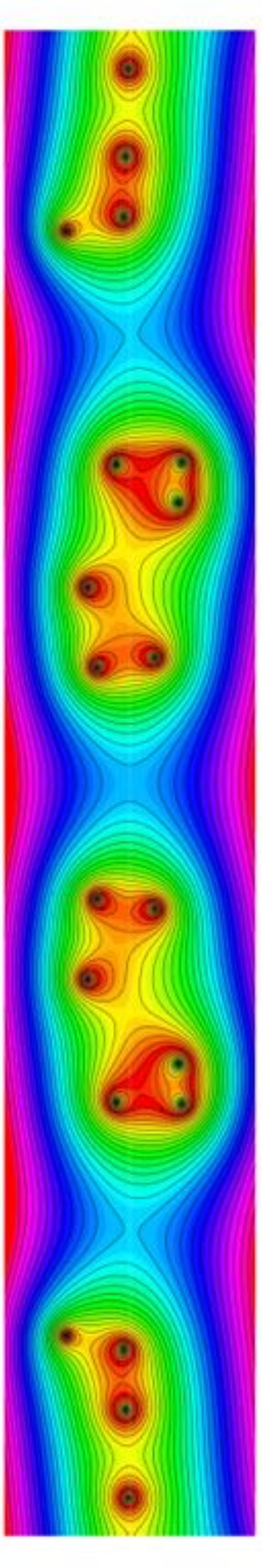}}
\scalebox{0.04}{\includegraphics{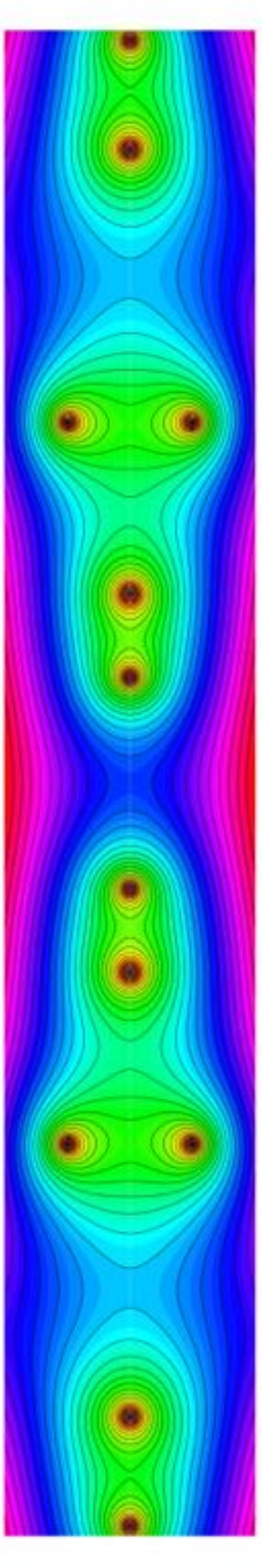}}
\scalebox{0.04}{\includegraphics{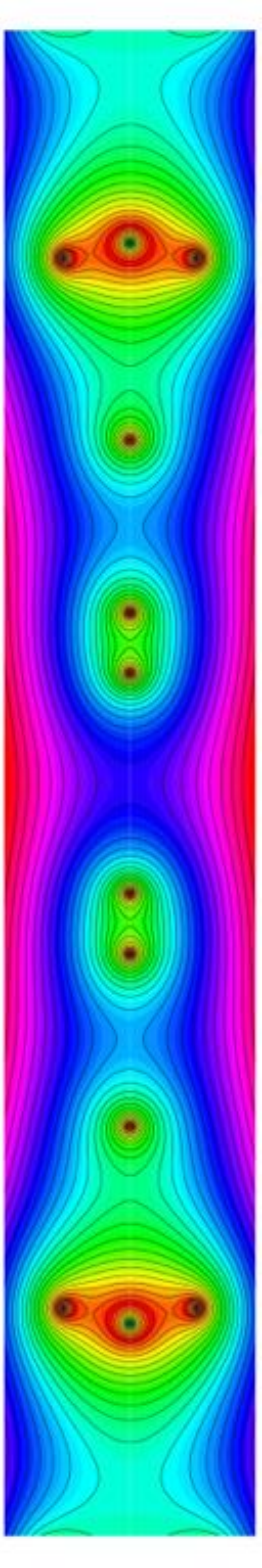}}
\scalebox{0.04}{\includegraphics{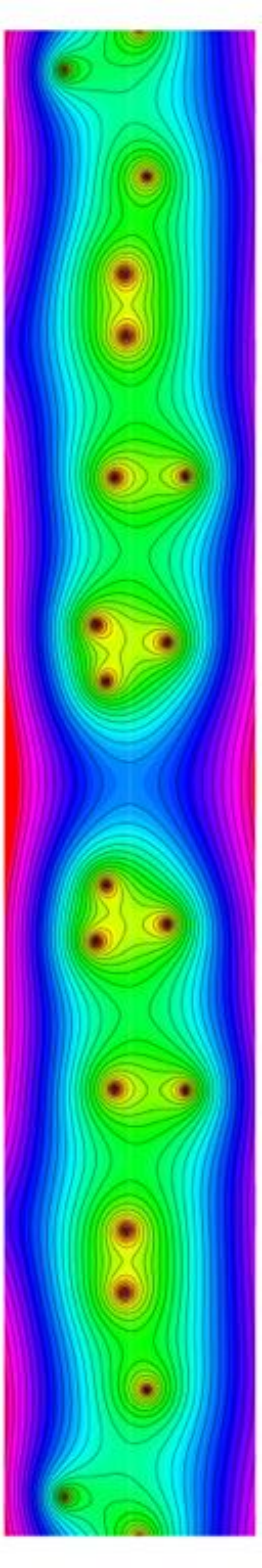}}
\caption{
The zeta function of the 5 platonic solids Tetrahedron, Octahedron, Cube, Dodecahedron and Icosahedron
(equipped with the Whitney complex). The cube and Dodecahedron are both $1$-dimensional 
(as the graphs have no triangles) and so satisfy the spectral symmetry. 
}
\end{figure}

\paragraph{}
Palindromic polynomials are also called self-reciprocal. They are of some
importance in other parts of mathematics. Here are some connections:
they play a role in coding theory \cite{JoynerShaska}, symplectic
geometry or knot theory. A symplectic polynomial is the characteristic polynomial 
of a symplectic matrix. A result of Seifert assures that a polynomial is the 
Alexander polynomial of some knot if and only if it is monic and reciprocal and 
$p(1)= \pm 1$ \cite{ChrisSmyth2015}. Even degree monic integer polynomials are
exactly the symplectic polynomials \cite{Margalit}. Since $(1+x)$
always is a factor if the degree is odd,
we can always write $p(x) = x^m q(x+1/x)$ or $p(x) = (1+x) x^m q(x+1/x)$ for a
polynomial $q$ of degree $m$. This illustrates a general
fact that one can compute the roots of palindromic polynomials up to degree 9 with
explicit formulas \cite{Lindstrom}. Also every minimal polynomial of a Salem number is
palindromic \cite{ChrisSmyth2015}.
%  {a,b} = x /. Solve[y==x+1/x,x]

\paragraph{}
The complex 
$G=\{\{1\},\{2\},\{3\},\{4\},\{1,2\},\{1,4\},\{2,3\},\{3,4\}\}$
belongs to the circular graph $C_4$. 
\begin{comment}
$L = \left[ \begin{array}{cccccccc}
 1 & 0 & 0 & 0 & 1 & 1 & 0 & 0 \\
 0 & 1 & 0 & 0 & 1 & 0 & 1 & 0 \\
 0 & 0 & 1 & 0 & 0 & 0 & 1 & 1 \\
 0 & 0 & 0 & 1 & 0 & 1 & 0 & 1 \\
 1 & 1 & 0 & 0 & 1 & 1 & 1 & 0 \\
 1 & 0 & 0 & 1 & 1 & 1 & 0 & 1 \\
 0 & 1 & 1 & 0 & 1 & 0 & 1 & 1 \\
 0 & 0 & 1 & 1 & 0 & 1 & 1 & 1 \\
\end{array} \right]$. 
\end{comment}
The characteristic polynomial of $L^2$ is 
$p(x)=x^8-32 x^7+316 x^6-1248 x^5+1926 x^4-1248 x^3+316
   x^2-32 x+1$. The quartic polynomial 
$q(x) = x^4-32 x^3+312 x^2-1152 x+1296$ has the property that 
$p(x) = x^4 q(x+1/x)$. This illustrates that one can compute the
roots of a palindromic polynomials up to degree 9. \cite{Lindstrom}

% COMMENT 4:

\paragraph{}
An amusing thing happens for cyclic polynomials $p$ of $G_n=C_{2^n}$:
the polynomial of $G_{n+1}=C_{2^{(n+1)}}$ has the polynomial $p_n$ 
of $C_{2^n}$ as a factor. Also $q_{n+1}/q_n$ is a square of an irreducible
polynomial of degree $2^n$. This means that the eigenvectors and eigenvalues
of $G_n$ have an incarnation as eigenvalues and eigenvectors in 
the Barycentric refinement $G_{n+1}$. But this happens only for circular
graphs, not for general $1$-dimensional complexes. 

% COMMENT 5:

\begin{figure}
\scalebox{0.32}{\includegraphics{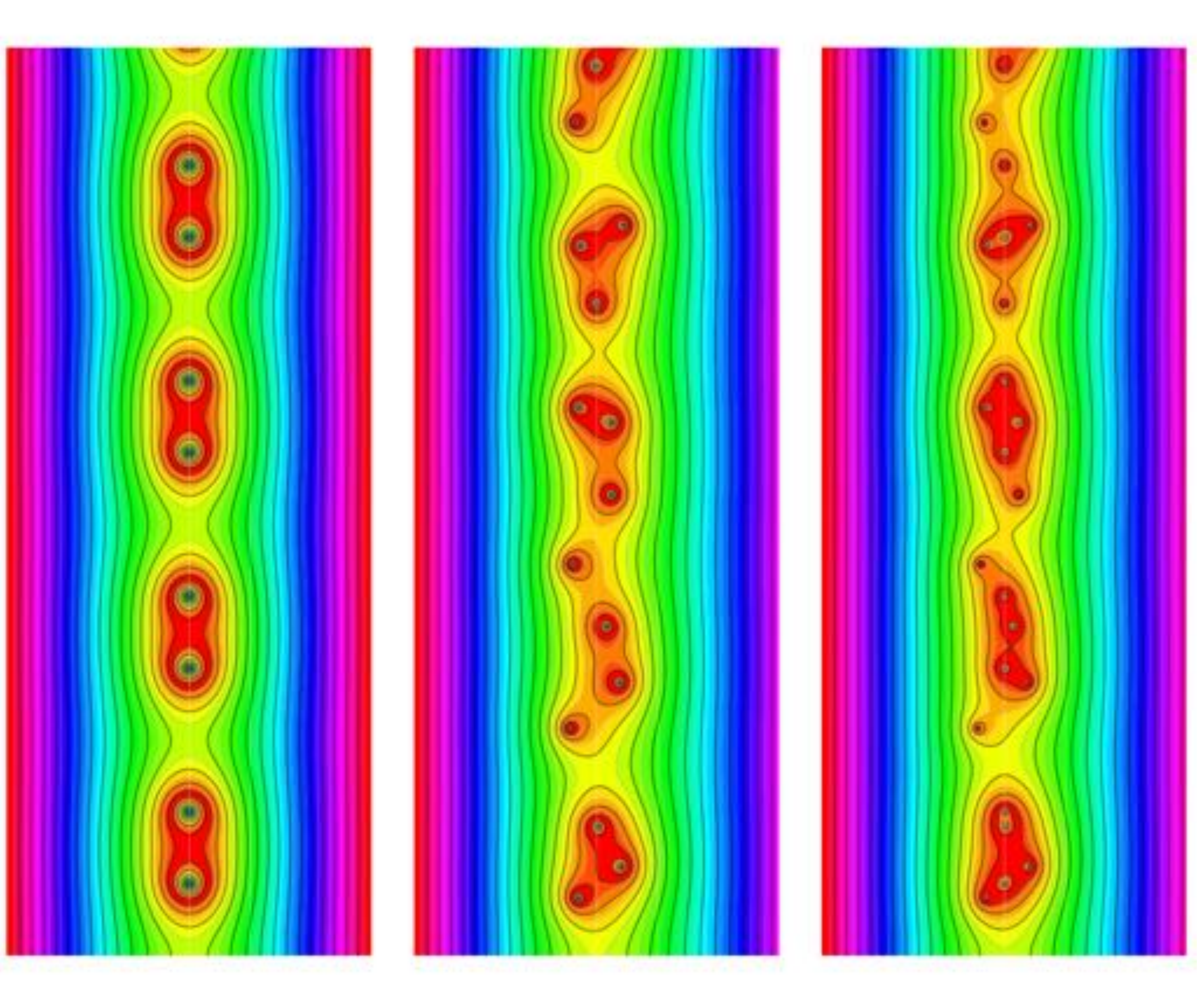}}
\caption{
The figure shows the contour plots of the zeta functions of the complete complexes $G=K_2$ and $H=K_3$ 
as well as of the product $G \times H$. The product is a $3$-dimensional $CW$ complex. 
The roots of the $\zeta_{G \times H}$ is exactly the union of the roots
of $\zeta_G$ and $\zeta_H$. 
}
\end{figure}

\paragraph{}
It would be nice to have other classes where one can get limiting
statements. One class to consider are the complete graphs $K_n$ which 
are $n-1$-dimensional complexes with $|G|=2^{n}-1$ simplices. 
We observe that $L(K_n)$ has only $2n-1$ different eigenvalues. This
could explain the periodic clumping along curves of the roots of 
the corresponding zeta functions. We also observe experimentally
that the coefficients of the characteristic polynomial of $L^2(K_n)$ 
is asymptotically nice and smooth. The case of the $1$-dimensional 
skeleton of $K_n$ is illustrated at the end of the article. 

\begin{figure}
\scalebox{0.2}{\includegraphics{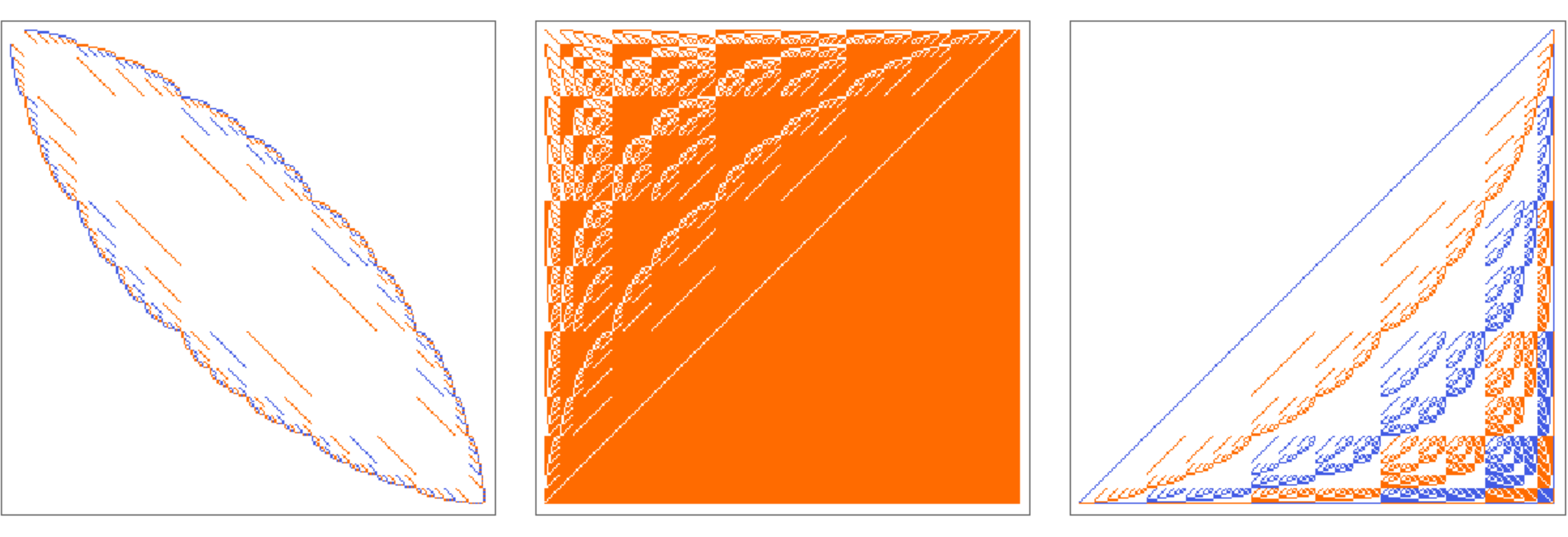}}
\scalebox{0.08}{\includegraphics{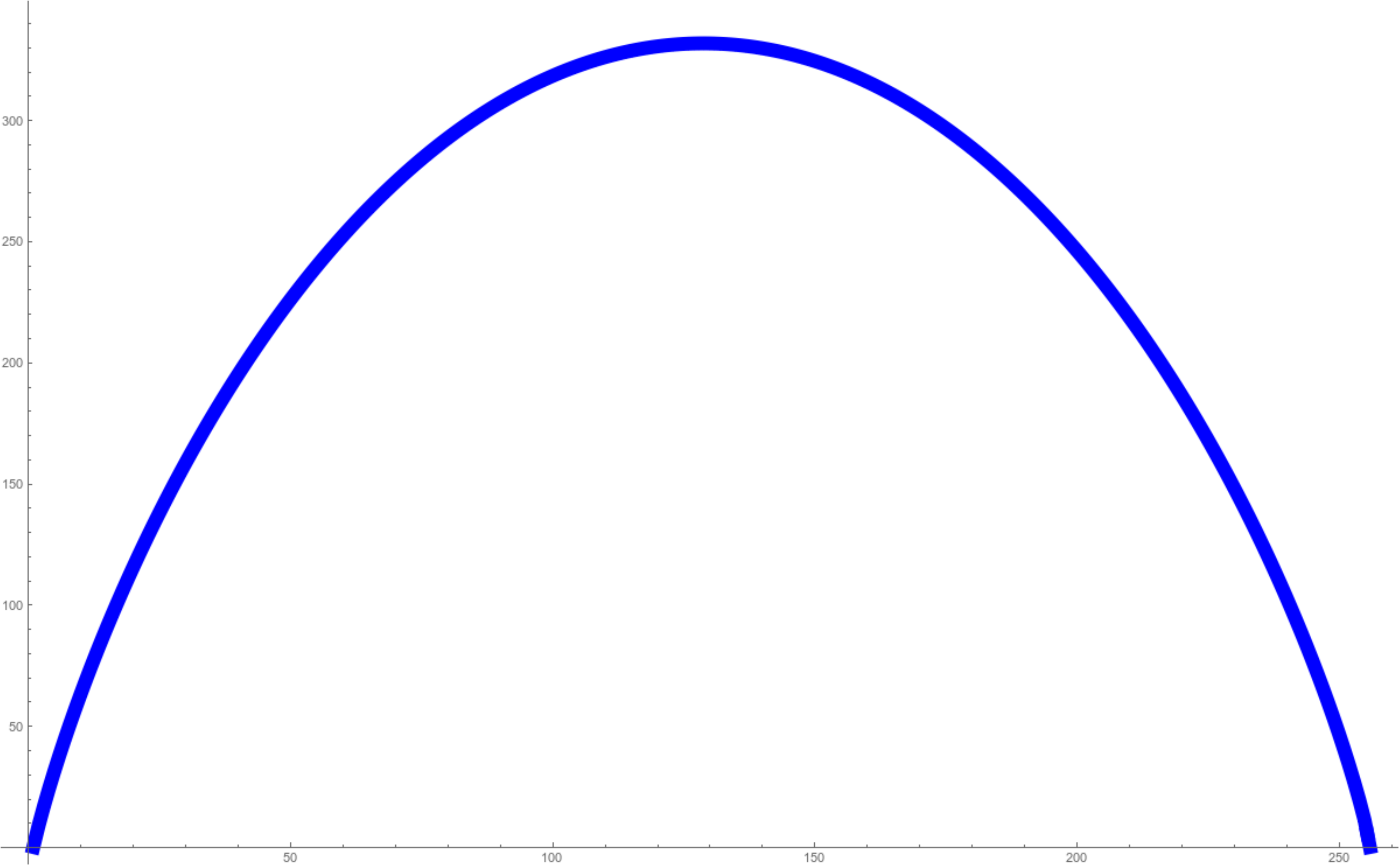}}
\caption{
The Dirac matrix $D$, the connection Laplacian $L$ and its inverse $L^{-1}$
for $G=K_8$ are $(2^8-1) \times (2^8-1)$ matrices. 
The last graph shows the coefficient list $k \to \log|p_k(L^2(K_8))|$ of the 
characteristic polynomial of $L^2$. 
}
\end{figure}

\section{Circular graphs}

\paragraph{} 
Like in the Hodge case, also for connection operators everything is explicit if the 
complex is a circular graph. 
The eigenvalues $\lambda_j$ of the connection Laplacian $L$ are obtained from the
eigenvalues $\mu_j$ of the Hodge Laplacian $H$ by a multiplicative
symmetrization $\lambda_j-1/\lambda_j = \mu_j$: 

\begin{lemma}
The connection Laplacian $L$ of a circular graph $C_n$ is a $2n \times 2n$
matrix $L$ which has the eigenvalues $\lambda_k = f^{\pm}(\mu_k)$, where
$\mu_k(k/n)$ with $\mu(x)=f^{\pm}(4 \cos^2(\pi x))$ and  
$f^{\pm}(x) = (x \pm \sqrt{x^2+4})/2$.
\end{lemma}

\begin{proof}
The matrix $M=L-L^{-1}$ has the eigenvalues
$\mu_k = \lambda_k - \lambda_k^{-1}$. If the simplices of the complex
are ordered as $G=\{ \{1\},\{1,2\},\{2\}, \dots, \{n-1,n\},\{n\},\{n,1\} \}$,
then $M=2+Q+Q^*$, where $Q u(n)=u(n+2)$ is shift by $2$.
Fourier theory gives the eigenvalues $2+2\cos(2\pi k/n) = 4 \cos^2(\pi k/n)$.
\end{proof} 

\paragraph{} As in the Hodge case, the eigenvalues of $M$ are 
$\{ 4 \sin^2(\pi k/n) \}_{k=1}^n$ but appear with algebraic multiplicity $2$ at 
least each. In the Hodge case, the limiting density of 
states is the equilibrium measure on the real line segment $I=[0,4]$ in the complex plane. We can 
see the density of states of the connection Laplacian as the pull back of
the map $T(z) = z-1/z$. It is the equilibrium measure on the union of two intervals
$J = [1,2+\sqrt{5}]$ and $[-1,-2+\sqrt{5}]$. The map $T$ maps the intervals $J$ onto 
the single interval $I$. 

% COMMENT 6

\paragraph{}
The list of eigenvalues of $L$ is the same when taking $4 \sin^2(\pi k/n)$ rather
than $4 \cos^2(\pi k/n)$. This is more convenient as it relates directly to the Hodge case: 
where the eigenvalues were $\lambda_k = g(k/n)$ with $g(x) = 4 \sin^2(\pi x)$.
We see that the map $T(x)=x+1/x$ maps the density of states measure on $\sigma(L)$ 
to the density of states measure of $\sigma(H)$. 

% COMMENT 7

\section{Pythagoras} 

\paragraph{}
Given two arbitrary $m \times n$ matrices $F,G$, the coefficient $p_k$ of the 
characteristic polynomial of the $n \times n$ matrix $F^T G$ satisfies the identity
$$  p_k = \sum_P {\rm det}(F_P G_P) \; , $$
where $F_P$ is a sub-matrix of $F$ in which only rows from $I \subset \{1,\cdots, n\}$
and columns from $J \subset \{1,\cdots, n\}$ are taken. 
This generalized Cauchy-Binet result was first found and proven in
\cite{CauchyBinetKnill} and is the first formula for the 
characteristic polynomial of a product of arbitrary two matrices. The classical
Cauchy-Binet theorem is the special case for $p_k$ when $k={\rm min}(m,n)$, where
${\rm det}(F^T G) = \sum_{|P|=k} {\rm det}(F_P) {\rm det}(G_P)$.
In the even more special case $k=n=m$, it is the familiar product formula for determinants.
A consequence of the generalized Cauchy-Binet result \cite{CauchyBinetKnill} is:

\begin{lemma}[Generalized Pythagoras]
Given a matrix $L$, then 
$$  p_k(L^T L) = \sum_{|P|=k} {\rm det}((L^T)_P) {\rm det}(L_P) \; . $$
\end{lemma}

\paragraph{}
Note that $(L^T)_P$ is not the same than $(L_P)^T$. 
The formula for the coefficients $p_k$ implies for a self-adjoint matrix $L$, 
that the coefficients $p_k$ of the characteristic polynomial of $L^2$ satisfies
some Pythagorean identities. It implies also the formula $p_k(A) = {\rm tr}(\Lambda^k A)$ 
\cite{SimonTrace}. In the even more special case $k=1$, where $p_1(A)$ is the trace, 
one has then the Hilbert-Schmidt identity 
${\rm tr}(L^T L) = \sum L_{ij}^2$ is a version of Pythagoras
justifying the name ``generalized Pythagoras". 

\paragraph{} 
In \cite{HearingEulerCharacteristic} we have looked at the deformation
$$ \tilde{K}(t) = \left[ \begin{array}{ccccccc}
    L_{11} & L_{12} & . & . & . & L_{1n} & t L_{1,x}\\
    L_{21} & L_{22} & . & .  & . & . & t L_{2,x} \\
     . & . & . & . & . & . & t L_{3,x} \\
     . & . & . & . & . & . & .   \\
     . & . & . & . & . & . & .   \\
     . & . & . & . & . & . & .   \\
    L_{n1} & . & . & . & . & L_{nn} & t L_{n,x} \\
   1 & 1 & 1 & 1 & 1 & 1 & 1 \\
\end{array} \right] . $$ 
We look here at the more symmetric deformation
$$ K(t) = \left[ \begin{array}{ccccccc}
    L_{11} & L_{12} & . & . & . & L_{1n} & t L_{1,x}\\
    L_{21} & L_{22} & . & .  & . & . & t L_{2,x} \\
     . & . & . & . & . & . & t L_{3,x} \\
     . & . & . & . & . & . & .   \\
     . & . & . & . & . & . & .   \\
     . & . & . & . & . & . & .   \\
    L_{n1} & . & . & . & . & L_{nn} & t L_{n,x} \\
    t L_{x1} & t L_{x2} & \dots & \dots & \dots & t L_{xn}& 1  \\
\end{array} \right] . $$
It interpolates $L=K(0)$, the connection Laplacian of the complex $G$ with 
$K(1)$, the connection Laplacian of the complex $G \cup_{A} \{x\}$ in which a
new cell $x$ is attached to $G$ along $A$. As for $\tilde{K}(t)$ we have
${\rm det}(K(t)) = {\rm det}(K) (1-2t^2)$ which is consistent with the fact that adding
an odd-dimensional simplex changes the sign of the determinant of the connection
matrix $L$, a fact which is true for all simplicial complexes. 

\begin{lemma}[Coefficients of characteristic polynomial]
The coefficients $q_k(t)$ of the characteristic polynomial 
of $K_{ij}(\sqrt{t})$ are all linear $t$. The coefficients $p_k(t)$ of the 
characteristic polynomial of $K_{ij}(t)^2$ are of maximal degree $4$ in $t$. 
\end{lemma}
\begin{proof}
This follows from the Pythagoras formula and the fact that every 
minor ${\rm det}(K_P(t))$ is linear or quadratic in $t$. The reason for the later
is that every $1$-dimensional path the directed sub-graph defined by $P$ 
can enter or leave $x$ only once. It also can be seen directly when doing the
Laplace expansion of the determinant with respect to the last column. 
Therefore, each term ${\rm det}(K_P(t))^2$ is a mostly
quartic function in $t$ and only coefficients $c_1 + c_2 t^2 + c_4 t^4$ appear. 
\end{proof} 

Now there are lots of such functions starting with the value $0$ at $t=0$,
it appears like a miracle that all these functions have again a root at $t=1$. But
this is exactly what happens in the $1$-dimensional case.
This will be explained in the next section. 

% COMMENT 8

\section{Proof of the functional equation}

\paragraph{}
In order to show the spectral symmetry it suffices to prove the
palindromic relation
$$  p_k = \sum_{|P|=k} {\rm det}(L_P)^2 = \sum_{|P|=n-k} {\rm det}(L_P)^2  = p_{n-k} \; . $$

Every pattern $P=I \times J$ of size $k$ corresponds to a directed $1$-dimensional closed 
directed sub-graph of $G$. The formula for $p_k$ is then a ``path integral" over all 
directed paths of length $k$. 
Using the complementary pattern $\overline{P}=\overline{I} \times \overline{J}$ satisfying 
$I \cup \overline{I}=J \cup \overline{J}=\{1, \dots, n\}$, and
$I \cap \overline{I}=J \cap \overline{J}=\emptyset$, we can write the claim also as
$$  \delta_k(t) = \sum_{|P|=k} {\rm det}(L_P)^2 - {\rm det}(L_{\overline{P}})^2 \; . $$

What happens if we deform the operator $L$ to an augmented one using the 
operator $K(t)$? While the palindrome difference deformation function
$$  \delta_k(t) =  p_k(K(t))-p_{n-k}(K(t)) $$ 
is not constant zero for $t \in [0,1]$, we can show that $\delta_k(t)$ is
zero for $t=1$ again, proving thus the palindromic property.

\begin{propo}[Artillery Proposition] 
If $G$ is a $1$-dimensional simplicial complex and a new edge $e=(a,b)$
is added to $G$, leading to the deformation operator $K(t)$, then each 
palindrome difference function satisfies 
$$    \delta_k(t) = C_k  t^2 (1-t^2) \; , $$
where $C_k$ is an integer.
\end{propo}

\paragraph{}
The connection Laplacian is $L=I+A$, where $I$ is the identity matrix and $A$ is the
adjacency matrix of the connection graph $G'$ of $G$, the graph which has the simplices of $G$
as vertices and where two vertices $x,y \in G$ are connected if they intersect. 
As used in \cite{Unimodularity}, the determinant of $L$, the Fredholm determinant of $A$
can be interpreted as a signed sum over all possible $1$-dimensional oriented closed subgraphs 
(unions of closed oriented circular paths), where fixed points are counted as closed paths of length $1$  
and closed loops $a \to b \to a$ along an edge $e=(a,b)$ are counted as closed 
paths of length $2$. The sign of each path is according to whether it has even or odd length. 
This is in accordance to the Leibniz definition of determinants; every oriented closed path
corresponds to permutation in the determinant and the decomposition into different connectivity 
components exhibit the cycle structure of the permutation. 
Figures~(14)-(17) in \cite{Unimodularity} illustrate this geometric interpretation
for concrete graphs, where all oriented closed subgraphs are drawn.

\paragraph{}
The coefficient $p_k$ of the characteristic polynomial of $L^2$ also has a path integral interpretation.
It depends on the pattern $P=I \times J$ however, whether paths in it can be closed or not. If $I=J$, then 
every connected component in the path is closed. If $I \cap J =\emptyset$, then no closed path components 
appear. The formula $p_k(L^2)=\sum_{|P|=k} {\rm det}(L_P)^2$ shows that we have to look 
at all paths of length $k$ now. The length of a connected path component $\{ v_1,v_2, \dots, v_m\}$ is is $m$
if $v_1 \neq v_m$ and $k-1$ if $v_m=v_1$, a case where the path is closed. 
A vertex $\{ v \}$ alone is considered a path of length $1$ and a path $\{v_1,v_2,v_1 \}$ is a path of 
length $2$. The reason for these assumptions is it that the connection matrix $L$ has $1$ in the
diagonal so that, if it is interpreted as an adjacency matrix, we are allowed to loop at a simplex, but
then that path of length $1$ is isolated. 
Paths contributing to a minor in $P=I \times J$ with $|P|=k$ are given by a collection of $k$ tuples
$\{v_1, \dots, v_k\}$ with $k$ different elements $v_j$ where transitions
$v_i \in I \to v_{i+1} \in J$ or $v_k  \in I\to v_1 \in J$ are allowed. We can think of $I$ as the exit
set and $J$ the entry set. This allows especially transitions $v_i \to v_i$ if $i \in I \cap J$. 

\paragraph{}
There are three type of patterns $P$: \\

Class A) $P=I \times J$ satisfies $e \notin I, e \notin J$. It represents paths of length $k$
which never enter, nor leave $e$. \\

$\overline{A}$) Patterns $\overline{P}$, where $P$ is in A).  \\

Class B) $P=I \times J$ has the property that exactly one of the subset $I,J$ contains $e$. 
These paths add to the $t^2$ contributions in the polynomial $p_k(t)$. \\

$\overline{B}$) Patterns $\overline{P}$ where $P$ is in $B)$.  \\

Class C) $P= I \times J$ has the property that $e \in I, e \in J$. This produces the $t^4$
contributions in $p_k(t)$. \\

$\overline{C}$) Patterns $\overline{P}$, where $P$ is in C). 

\paragraph{}
Now we can use induction with respect to the number of edges. There are two things to show:
the constant part is zero and the $t^2$ and $t^4$ contributions have opposite sign. 
The induction assumption is the case when $G$ is zero dimensional with no edge. 
But then, $K$ is the identity matrix and the coefficients of the characteristic polynomial 
are Binomial coefficients which are palindromic. The induction foundation is established. \\

{\bf (i)} All the constant parts go away: the paths not hitting the edge $e$ in 
class A) and class $\overline{A}$ cancel. We have to show that the contribution of patters
$|P|=k$ which $e \notin P$ is equal to the contribution of $|\overline{P}|=k$
never hiting $e$. We can formulate this as $A - \overline{A} = 0$.  \\

{\bf Proof of i):} If $P=I \times J$ has $e \notin I, e \notin J$, we can take $e$ 
away and look at the complex $G$ before adding $e$, where we know the statement 
already to hold. 

\paragraph{}
{\bf (ii)} To see the symmetry between $t^2$ and $t^4$ we have to show that 
the contribution of paths of length which either leave or only enter $e$ minus the 
contribution of paths of length $n-k$ which either only leave or only enter $e$
agrees with the difference of the number of paths of length $k$ which leave and enter $e$,
minus the contribution of paths of length $n-k$ which leave and enter $e$.
We can formulate this as $B-\overline{B} = C - \overline{C}$.  \\

{\bf Proof of ii):} There exists an other edge $f$ in $G$ different from $e$. Otherwise, we are
in the induction foundation case. We split now the sum
$$ \sum_{|P|=k} {\rm det}(L_P)^2 - {\rm det}(L_{\overline{P}})^2  $$
into three  parts. Look first at all patterns $P=I \times J$, where $f \in I, f \in J$. 
For those patterns the $t^2$ and $t^4$ contributions are the same by induction: 
it is the sum when the edge $f$ has been taken away. 
Now look at all the patterns $P=I \times J$, where $f \notin I, f \notin J$. 
But this is the same sum as before just with a negative sign, and it is therefore again symmetric in 
$t^2$ and $t^4$. Finally, look at 
$$  U = \sum_{|I \times J|=k, f \in I, f \notin J} {\rm det}(L_P)^2 - {\rm det}(L_{\overline{P}})^2 $$
and 
$$  V = \sum_{|I \times J|=k, f \notin I, f \in J} {\rm det}(L_P)^2 - {\rm det}(L_{\overline{P}})^2 $$
These two add up to zero $U+V=0$ because of symmetry.  \\

It follows that $p_k(L^2) = 0 + c t^2 - c t^4$ for some 
constant $c$, which is the claim of the artillery proposition.  \\

\paragraph{}
{\bf Remarks:} \\
{\bf a)} Where did the assumption enter that
$G$ is $1$-dimensional? It was in the induction part, 
where the one dimensionality of the complex mattered. 
When we add the first two-dimensional cell and omit one of the edges, then 
we don't deal with a simplicial complex any more. \\
{\bf b)} We can look in general at the deformation of the Green
function entries $g_t(x,y) = K(t)^{-1}(x,y)$. Using the Cramer's formula, 
one can see that $g_t(x,y) (1-2t^2)$ is a quadratic polynomial in $t$ for
all $x,y$. Since the Euler characteristic changes by $-1$ as we add one single edge
also the total energy $\sum_{x,y} g_t(x,y)$ changes by $-1$. In the one
dimensional case, we have complete control about the Green functions and
Green function changes. This will also allow us to see that $L-L^{-1}$ is the 
Hodge Laplacian. 

\begin{figure}
\scalebox{0.7}{\includegraphics{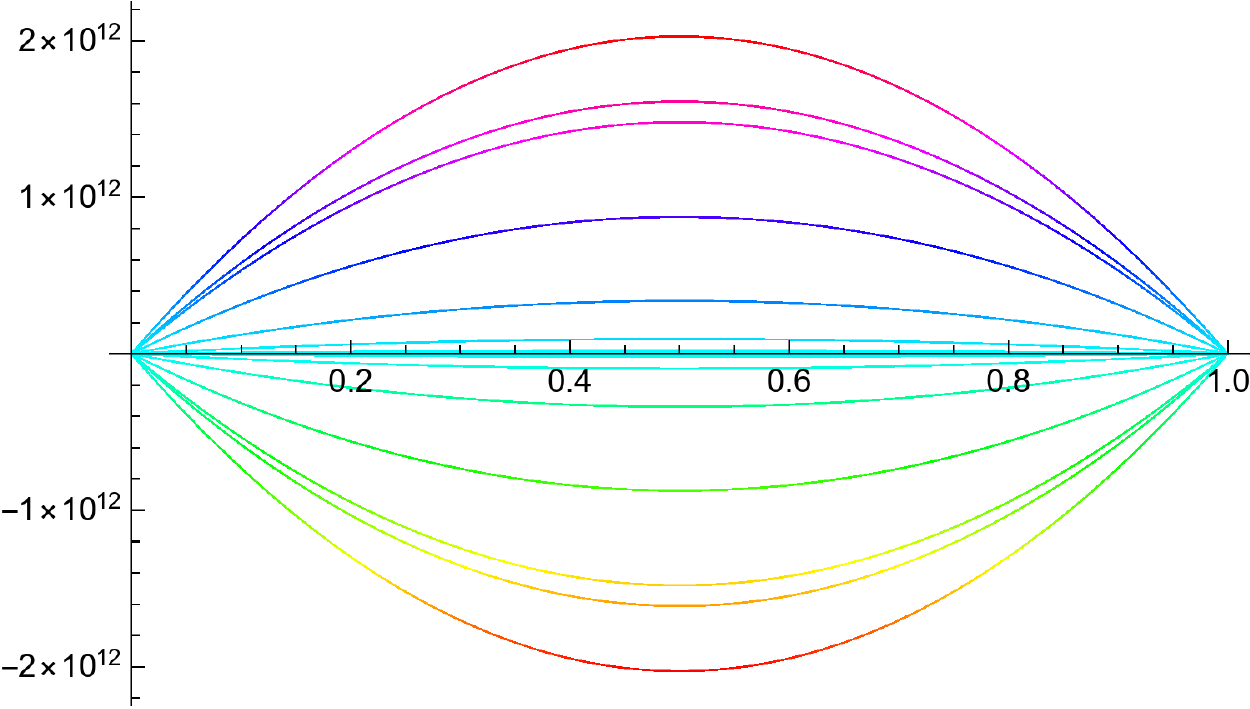}}
\caption{
We see the palindromic difference $\delta_k(t)=p_k(K(\sqrt{t}))-p_{n-k}(K(\sqrt{t}))$ 
of the coefficients of the characteristic polynomial of the deformation 
$K(t)$, when adding an other edge.
The artillery proposition assures that all trajectories hit the ground 
simultaneously at $t=1$. When taking time $\sqrt{t}$ we see parabolas. 
}
\end{figure}

\paragraph{}
While all $\delta_k(\sqrt{t})$ are quadratic functions and the induction 
assumption assures $\delta_k(0)=0$ for all $k$, we have for a general
complex (dropping the assumption that it is necessarily $1$-dimensional), 
that the second root of $\delta_k(t)$ depends on $k$. 
For $1$-dimensional complexes however, $\delta_k(1)=0$ 
for all $k$ assuring that also the enlarged complex keeps the 
palindromic property. For general complexes, the choreography falls out
of sync. We need to deform the operator $L$ to synchronize them again. 

\paragraph{}
{\bf Example:} take the complex $G=\{\{1\},\{2\},\{3\},\{1,2\}\}$ which is not connected
and has a single vertex not connected to a $K_2$ component. We have 
$L=\left[ \begin{array}{cccc} 1 & 0 & 0 & 1 \\ 
                              0 & 1 & 0 & 1 \\
                              0 & 0 & 1 & 0 \\ 
                              1 & 1 & 0 & 1 \\ \end{array} \right]$. Now, we add
the new edge connecting the single vertex $3$ to the rest. This gives the deformation
$K(t) = \left[ \begin{array}{ccccc}
 1 & 0 & 0 & 1 & 0 \\
 0 & 1 & 0 & 1 & t \\
 0 & 0 & 1 & 0 & t \\
 1 & 1 & 0 & 1 & t \\
 0 & t & t & t & 1 \\
\end{array} \right]$. Now, $K(1)$ is the connection Laplacian of 
$H=\{\{1\},\{2\},\{3\},\{1,2\},\{2,3\}\}$. The list of coefficients of the
characteristic polynomial of $K(0)$ is the palindrome $(1,9, 22,22, 9, 1)$ 
while the list of coefficients of the characteristic polynomial of $K(1)$
is the palindrome $(1, 15, 49, 49, 15, 1)$.  

\paragraph{}
To illustrate the computation of coefficients $p_k$, let us take the simpler example of the
transition $G=\{\{1\},\{2\}\}$ to $G_e$= $\{\{1\},\{2\},\{1,2\}\}$, where 
$K(t) = \left[ \begin{array}{ccc} 1 & 0 & t \\ 0 & 1 & t \\ t & t & 1 \\ \end{array} \right]$. 
We have $p_{K(t)}(x) = 1 (-x)^3 + (3+4t^2) x^2 + (3+4t^2) (-x) + (1-2t^2)^2$
so that the coefficients deform as $p_0(t)=1, p_1(t)=p_2(t)=3+4t^2, p_3(t)=(1-2t^2)^2$. 
Now, as $n=3$ we have $\delta_0(t)=p_0(t)-p_n(t)=4t^2(1-t^2), 
\delta_1(t)=p_1(t)-p_2(t)=4t^2(1-t^2)$. This illustrates the artillery proposition even so
$G$ was $0$-dimensional. The patterns with $|P|=1$ are the $1 \times 1$ submatrices
$$ P \in \{ [1], [0], [t], [0], [1], [t], [t], [t], [t] \} \; . $$
The sum of the squares of the determinants is $3+4t^2$. 
The patters with $|P|=2$ are the $(2 \times 2)$-submatrices
$$
             \{   \begin{array}{ccc}
                   \left[ \begin{array}{cc}1&0 \\ 0&1 \end{array} \right], &
                   \left[ \begin{array}{cc}1&t \\ 0&t \end{array} \right], &
                   \left[ \begin{array}{cc}0&t \\ 1&t \end{array} \right], \\
                   \left[ \begin{array}{cc}1&0 \\ t&t \end{array} \right], &
                   \left[ \begin{array}{cc}1&t \\ t&1 \end{array} \right], &
                   \left[ \begin{array}{cc}0&t \\ t&1 \end{array} \right], \\
                   \left[ \begin{array}{cc}0&1 \\ t&t \end{array} \right], &
                   \left[ \begin{array}{cc}0&t \\ t&1 \end{array} \right], &
                   \left[ \begin{array}{cc}1&t \\ t&1 \end{array} \right] \} \; 
                  \end{array}
$$
of $K(t)$. It leads to the minors $1,t,-t,t,1-t^2,-t^2,-t,-t^2,1-t^2$ whose sum of the squares
is $3+4t^4$. Finally, there is one matrix $K_P=K$ if $|P|=3$ having the minor 
${\rm det}(K)=1-2t^2$ leading to $p_3(t)$. 

% COMMENT 9, see folder illustration

\section{Convergence}

\paragraph{}
Given a $1$-dimensional simplicial complex $G=G_0$, we can look at 
its $n$'th Barycentric refinement $G_n$. Let $\zeta_n(s)$ denote the normalized connection
zeta function of $G_n$. That is $\zeta_n(s) = \sum_k \lambda_k^{-s}$, where $\lambda_k$
are the eigenvalues of $L^2(G_n)$. Let $\mu(G_n)$ denote the zero locus measure of
the zeta function $\zeta_n$. It is a measure summing up the Dirac masses of the 
roots of $\zeta_n$. We say that $\mu_n$ converges weakly to the imaginary axes,
if for every compact set $K$ away from the imaginary axes, we have $\mu_n(K) =0$ for 
large enough $n$. In our case, the convergence is nicer as we have a definite entire limiting
function $z(s)$. 

\begin{thm}[The limiting roots] 
For every $1$-dimensional complex $G$, the roots of $\zeta_{G_n}$ converge weakly 
to the roots of an explicit limiting function $z(s)$.
The functions $\zeta_{G_n}/|G_n|$ converge to the entire function $z(s)$ uniformly on 
compact subsets. 
\end{thm}

\paragraph{}
For the analysis, it is enough to look at the circular case:
for every $1$-dimensional simplicial complex $G$ there
exists a finite integer $m(G)$ such that $m$ cutting or gluing procedures
produce a circular graph. 2) As shown in \cite{KnillBarycentric,KnillBarycentric2},
if a complex $G$ is modified to a complex $H$ by changing maximally $m$ columns and $m$ rows only,
then the eigenvalues $\lambda_j$ of $G$ and eigenvalues $\mu_j$ of $H$ satisfy 
$\sum_j |\lambda_j^2-\mu_j^2| \leq C m(G)$, where $m(G)$ depends on $G$ only and
$C$ does not depend on the choice of the Laplacian (like $C=4$ for the 
connection Laplacian of a $1$-dimensional complex). For a constant $C(s)$ such that
$|\sum_j |(\lambda_j^2)^s| - \sum_j |\mu_j^2)^s|| \leq C(s) m$. 
This follows from the Lidskii-Last inequality 
$||\mu-\lambda||_1 \leq \sum_{i,j=1}^{n} |A-B|_{ij}$ for the eigenvalues $\mu_j$ and $\lambda_j$
of two arbitrary selfadjoint matrices $A,B$ of the same size. 
This means $\zeta_{G_n}(s)/|G_n| - \zeta_{H_n}(s)/|H_n|$ goes to zero uniformly on
compact sets for $s$. 

\paragraph{}
For the circular case, one can give explicit expressions. Define the function 
$$ Z(x,s) = \left( f_+(4\sin^2(\pi x) )^2 \right)^s + \left( f_-( 4\sin^2(\pi x) )^2 \right)^s  \; , $$
where $f_{\pm}(y)$ are the solutions $x$ of the equation $x-1/x=y$. 
The function $1-1/x$ maps the two intervals $[-1,2-\sqrt{5}]$ and $[1,2+\sqrt{5}]$ to 
the interval $[0,4]$ and $f_{\pm}$ are the inverses. 
We have seen that the eigenvalues $\lambda_k^s$ of $L^2(C_n)$ 
satisfy $\{ \lambda_k^s+1/\lambda_k^s = Z(k/n,s) \}_{k=1}^n$. 
The normalized Zeta function $\zeta(s) = \sum_{k=1}^n \lambda_k^s$ of $C_n$ 
satisfies now 
$$ \zeta_{C_n}(s) = \sum_{k=1}^n Z(\frac{k}{n},s)  \; . $$
The roots of $\zeta_n(s)$ converge in the limit $n \to \infty$ to
the roots of $z(s)$. Indeed, $\zeta_n(s)/n \to z(s)$. 

%  COMMENT 10

\paragraph{}
We see that the zeta function is $n$ times a Riemann sum for the integral
$$ z(s) = \int_0^1 Z(x)^s \; dx \; . $$
We have seen in the Hodge case already that for smooth functions, the 
error $\sum_{k=1}^n f(k/n) - n \int_0^1 f(x) \; dx$ goes to zero. This 
implies that if $\int_0^1 f(x) \; dx$ is non-zero, then for large enough $n$
the sum $\sum_{k=1}^n f(k/n)$ can not be zero. \\

\paragraph{}
Using a change of variables $v=\sin(\pi x)^2$ we can write
$$ z(it) = \int_0^1 {\rm Re} \frac{2 \left(\sqrt{4 v^2+1}+2 v\right)^{2 i t}}{\pi  \sqrt{1-v} \sqrt{v}} \, dv  \;  $$
or in the form d) of theorem~(\ref{limiting}).

% COMMENT 11

\begin{figure}
\scalebox{0.3}{\includegraphics{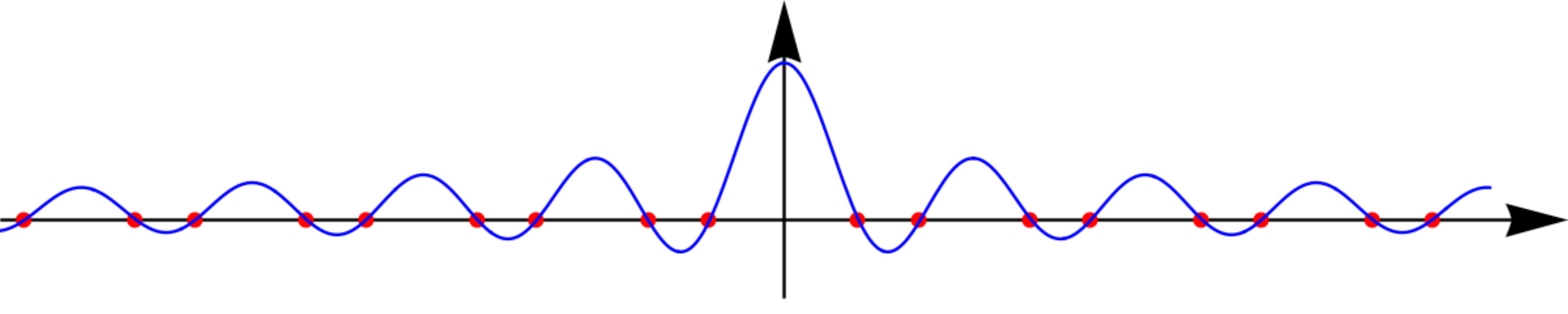}}
\caption{
The limiting zeta function is here restricted to the imaginary axes so that we can
plot the graph of $t \to z(it)$ on $t \in [-10,10]$ and mark the roots in 
that interval. 
}
\end{figure}

\paragraph{}
The limiting function $z(s)$ can be described better by turning the
picture by 90 degrees so that all roots become real.
The function $t \to z(it)$ is now a function which is real on the real axes. 
It is given in part c) of theorem~(\ref{limiting}). 
The roots of the limiting zeta function are the roots of this function. 

% COMMENT 12

\paragraph{}
{\bf Hypothesis:} All the roots of $t \in \CC \to w(t)$ are on the real line. \\

As this is an explicit generalized hypergeometric series, we don't expect
this to be difficult. 
We have just not found any reference yet, nor an argument 
which assures that. Let us thus express the limiting function in various ways:

\begin{thm}[limiting zeta function]
\label{limiting}
We can rewrite the function in various ways: \\
{\bf a)} Hypergeometric Series: $\zeta(2s)=\pi \, _4F_3\left(\frac{1}{4},\frac{3}{4},-s,s;\frac{1}{2},\frac{1}{2},1;-4\right)$.
$$ \zeta(2s,z) = \pi \sum_{n=0}^{\infty} \frac{(1/4)_n (3/4)_n (-s)_n,(s)_n}{(1/2)_n,(1/2)_n,(1)_n} \frac{z^n}{n!} $$
analytically continued to $z=-4$. \\
{\bf b)} Fourier Transform: we have a $\cos$-transform of a smooth measure on an interval
$$ \zeta(it) = \int_0^{\log(2+\sqrt{5}} \cos(t y) \sqrt{\frac{\cosh(y) \coth(y)}{2 - \sinh(y)}} \; dy \; . $$
{\bf c)} Original expression:
$$ \zeta(it) = \int_0^1 2\cos(t \log((\sqrt{4 \sin^4(\pi x)+1}+2 \sin^2(\pi x))^2 )) \; dx \; . $$
{\bf d)} Simplest integral
$$ \zeta(it) = \int_0^1 \frac{2 \cos \left(2 t \log \left(\sqrt{4 v^2+1}+2 v\right)\right)}{\pi  \sqrt{1-v} \sqrt{v}} \, dv \;  $$
{\bf e)} Abelian integral: if $a=2+\sqrt{5}$, then 
$$ \zeta(s) = \int_1^{a} \frac{\left(z^2+1\right) z^{-1-s} \left(1+z^{2 s}\right)}
                              {2 \sqrt{(1+z)(1-z)(z-a)(z-1/a)}} 
            = \int_1^a F(z) \; dz \; . $$

\end{thm}

\paragraph{}
\begin{proof}
The function $f(x) = \log(\sqrt{4 \sin^4(\pi x)+1}+2 \sin^2(\pi x))$
is a $1$-periodic, non-negative function with roots at $0$ and $1$. It is monotone from $0$ to $1/2$
For any smooth $1$-periodic non-negative function $f$ we can look at the complex function
$w(z) = 2 \int_0^{1/2} \cos(z f(x)) \; dx$. For general periodic $f$, the transformed function has 
non-real roots but in our case we can make a substitution: $u=f(x)$ gives
$x=-\arcsin(\sqrt{2 \sinh(u)}/2)/\pi = g(u)$. The integral is now
the $\cos$-transform $(1/(2\pi)) \int_0^{\log(2+\sqrt{5})} \cos(z u) h(u) \; du$, where $h$ is the function
$$ h(y)=\cosh(y)/\sqrt{(2-\sinh(y)) \sinh(y)} $$ 
on $I=[0,\log(2+\sqrt{5})] \subset [0,\pi/2]$. We can rewrite as in part {\bf b)} of the theorem. 
In other words, the limiting zeta function is $4\pi$ times the $\cos$-transform of the probability measure 
$\mu$ with density $h(y)$ on $I$. It is not true that for a general probability measure $\mu$ on the real 
line with compact support, the $\cos$-transform of $\mu$ is an entire function with only roots on the 
real axes. 

\paragraph{}
Let $a=2+\sqrt{5}$ and $p(x)=(1+z)(1-z)(z-a)(z-1/a)$. Then 
$$ \zeta(i t) = \int_1^{a} \frac{\left(z^2+1\right) z^{-1-i t} \left(1+z^{2 i t}\right)}{2 \sqrt{p(z)}}  \; dz 
   = \int_1^a F_z(s) \; dz \; . $$
Because $p$ is a polynomial of degree $4$ it is a generalized Abelian integral. Technically, for every 
integer $s=it$, it can be expressed as an Abelian integral integrating over a closed loop on the real 
elliptic Riemann curve $w^2=p(z)$ with $p(z)=(1+z)(1-z)(z^2-4z-1)$. (Nowadays, elliptic curves are mostly
written as cubic curves, in particular in Weierstrass form, but historically, the real quartic forms have
prevailed (by Abel, Jacobi, Gauss and even Euler earlier \cite{Alling}). There was a revival of quartic 
normal forms in particular in cryptology, like curves in Jacobi quartic form, where the addition formulas are 
particularly elegant). The polynomial $p$ can be written as 
$p(z)=z^2 ( (z-1/z)^2-4 (z-1/z))$. The elliptic curve not only has the anti-analytic symmetry
$(z,w) \to (\overline{z},\overline{w})$ but also features the analytic involution
$(z,w) \to (1/z,w/z^2)$. This involution will play a role in the future.
% X={w,(1+z)(1-z)(z^2-4z-1)}; Y=X /. w->w/u^2 /. z->1/u /. u->z; Z=Y /. w->w/u^2 /. z->1/u /. u->z; Z==X

\paragraph{}
It follows that for every integer $s$, the value is an elliptic integral. For every rational $s \in \QQ$,
we express the value as a hyperelliptic integral for a higher degree polynomial. For general $s$, it
is just the value of a generalized hyperelliptic function. The Abelian connection especially
implies that $\zeta(it)$ is a period for every rational $t \in \QQ$. This fact again contributes to the impression
that the Dyadic story is ``integrable" as Abelian integrals are everywhere in integrable dynamical 
systems. For the classical Riemann zeta function, we only know that the 
integers larger than $1$ are periods. (The class of periods is rather mysterious as one does not even know an
explicit example of a number (an explicit expression like $e^{\pi} + \sqrt{2}$) which is not a  
period evenso a cardinality argument shows that almost all are \cite{Periods}.)

% COMMENT 13

\paragraph{}
Here is an other piece of information which came up when trying to exclude
roots away from the axes of the hyperelliptic function.
The logarithmic derivative of the function $F_z(s)$ simplifies to
$$ \frac{F_z'(s)}{F_z(s)} = \log(z) (1-\frac{2}{(1+z^{2s})})  \; . $$
It would be nice to use some contour
integral to exclude roots away from the imaginary axes.
But the logarithmic derivative does not simplify any more nicely 
after the integral of $F_z$ over $z$ is performed. 

% COMMENT 14

\paragraph{}
As the polynomial $p(x)$ is palindromic, we can make a transformation $u=x+1/x$.  This gives 
$$  \sqrt{\frac{-(1 + z^2)^2}{1+4z-2z^2-4z^3+z^4}} \frac{1}{2z} dz = \frac{1}{\sqrt{4u-u^2}} du $$ 
and we get the integrals
$$ \int_0^4  ((u + \sqrt{4 + u^2})/2)^s/\sqrt{4u-u^2} \; du 
 + \int_0^4  ((u - \sqrt{4 + u^2})/2)^s/\sqrt{4u-u^2} \; du  \; . $$ 
This gives 
$$ \zeta(s) = \pi \, _4F_3\left(\frac{1}{4},\frac{3}{4},-\frac{s}{2},\frac{s}{2};\frac{1}{2},\frac{1}{2},1;-4\right) . $$
We use the notation of \cite{Hypergeometric,Koekoek}. By definition, this is the (analytically continued value) of the series 
$$ \zeta(2s,z) = \sum_{n=0}^{\infty} \frac{(1/4)_n (3/4)_n (-s)_n (s)_n}{(1/2)_n (1/2)_n (1)_n} \frac{z^n}{n!} \;  $$
for $z=-4$. It uses the Pochhammer notation $(a)_0=1,(a)_1=a,(a)_n=a(a+1)\cdots (a+n-1)$, which is also called a raising factorial.
Now, this is a sum of polynomials of the form $c_n (s)_n (-s)_n$ but the series is understood only
by analytic continuation; for all hypergeometric functions of type $_4F_3$,
the radius of convergence in $z$ is $1$. 
A hypergeometric function is by definition a sum 
$\sum_n c_n$, if $c_{n+1}/c_n$ is a rational function in $n$ \cite{Koekoek}. 
\end{proof}

% COMMENT 15

\section{Remarks}

\paragraph{}
The Barycentric limit of a $1$-dimensional complex is naturally the
dyadic group $\DD_2$ of integers. Similarly as the circle $\TT$ is the Pontryagin
dual of the integers $\ZZ$, the group $\DD_2$ is the natural dual to the 
Pr\"ufer group $\PP_2$, the group of $2$-adic rationals modulo $1$. The analogy
is to compare $\DD_2$ with $\TT$. Both are compact and a continuum,
having a dual which is not-compact and countable. But
$\DD_2$ is quantized in the sense that it features a smallest translation. It is a world
in which the Planck constant is ``hard wired" in. 
The arithmetic in the dyadic case is simpler as there are no interesting primes in the 
dual group $\PP_2$. It is no surprise therefore that the 
Zeta story is simpler. This realization also should squash any speculation that there
is a direct bridge between the "dyadic zeta function" and the "rational zeta function". 
As \cite{KnillZeta} and especially \cite{FriedliKarlsson} show in the Hodge Laplacian case,
the finite dimensional cases can approximate the Riemann zeta function and the speed of
approximation plays a role for the location of the roots. 

\begin{figure}
\scalebox{0.43}{\includegraphics{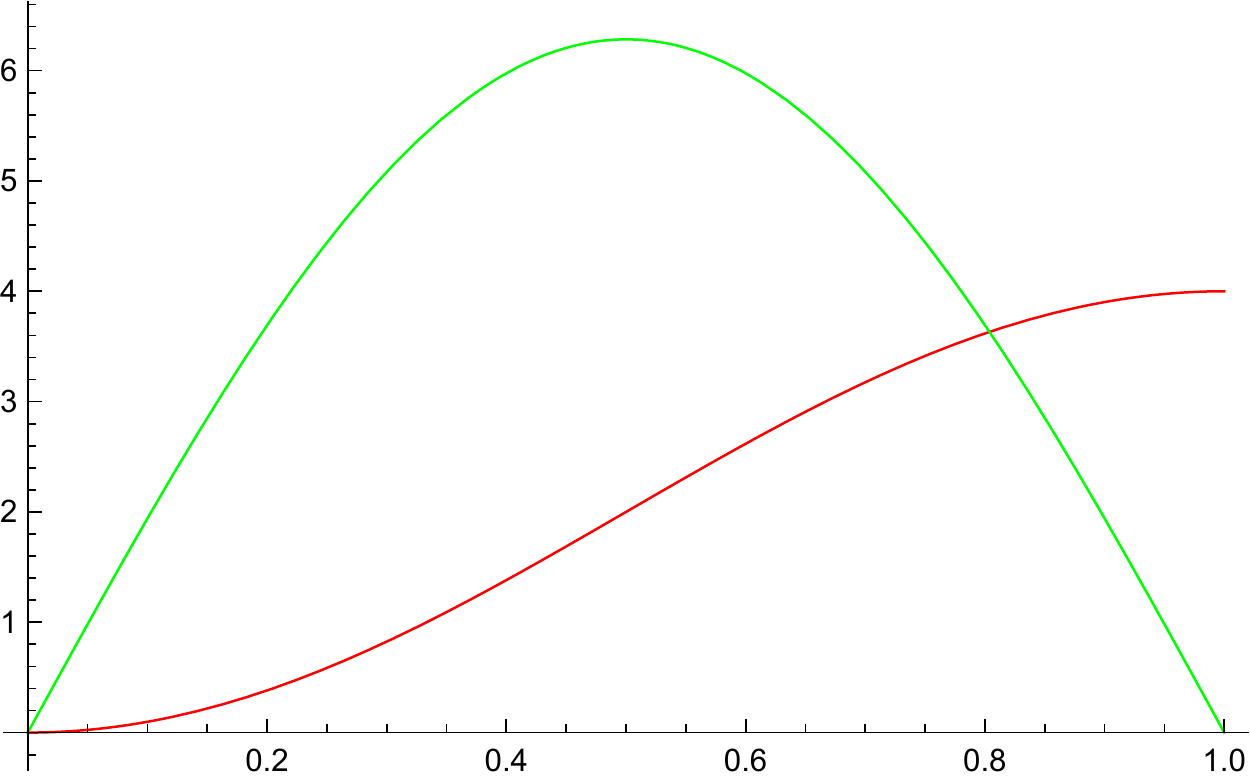}}
\scalebox{0.43}{\includegraphics{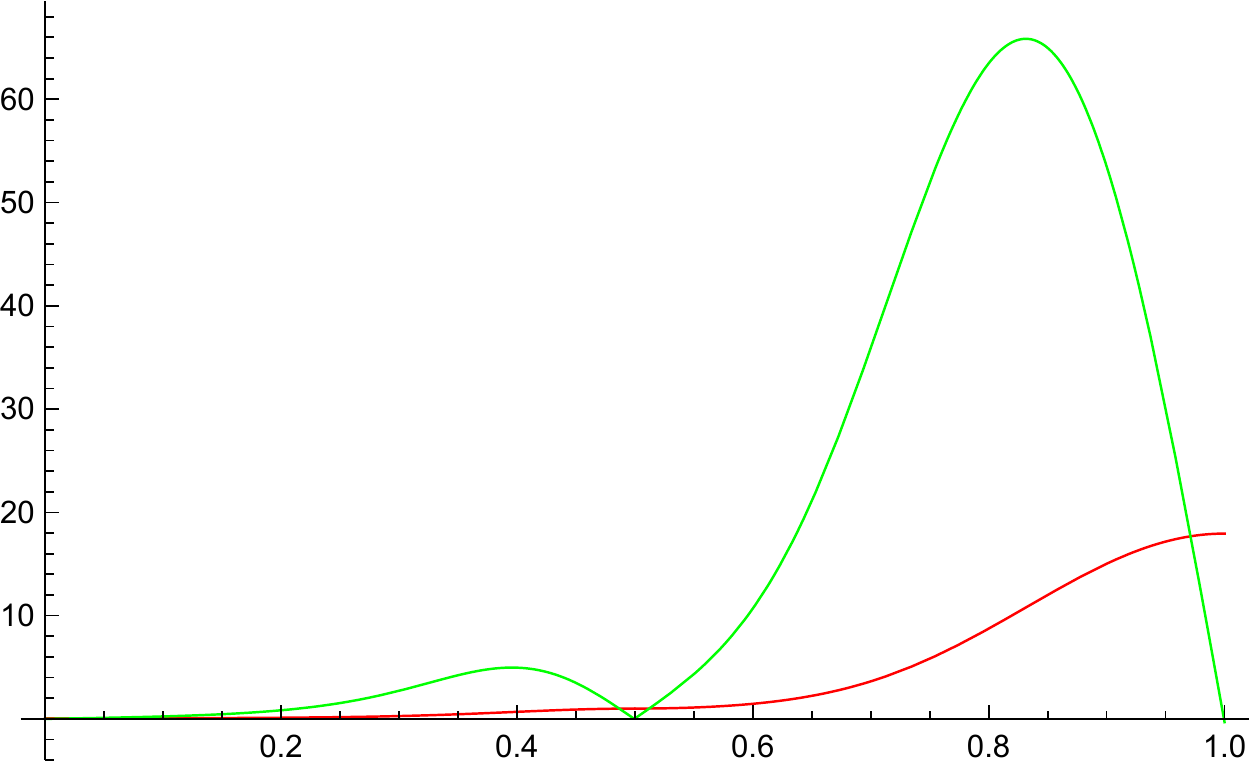}}
\caption{
The limiting function $f(x) = \lambda_{[n x]}$ and its derivative in the Hodge
case $H=D^2$ and then for the connection Laplacian case $L^2$. Both
limiting functions are exactly known.
}
\end{figure}

\paragraph{}
One can get to the dyadic picture via ergodic theory. 
When performing a Barycentric refinements of random Jacobi matrix $L$ defined over any
measure preserving dynamical system $(X,T,m)$, the renormalization map \cite{Kni95} allows
to write $L=D^2+c$ with an other Laplacian $D$ defined over a renormalized system
$(Y,S,n)$. The renormalization of the base dynamical system $T \to S$ is almost trivial as 
it is the 2:1 integral extension \cite{Friedman}, where one takes two copies of the probability space
and have $S$ map one to the other and then use $T$ to go back to the first. This map is a
contraction in a complete metric space of measure preserving dynamical systems
and has by the Banach fixed point theorem a unique fixed point $T$. This
von-Neumann-Kakutani system $T$ is also known under the name "adding machine".
The ergodic theory, $T$ is completely understood because the discrete spectrum in
ergodic theory allows to rewrite the system as a group translation. The Koopman operator $U$ on the Hilbert
space $L^2(\DD_2)$ has discrete spectrum which is the Pr\"ufer group. By general facts in 
ergodic theory \cite{CFS}, the system is now measure theoretically equivalent to a group translation 
on $\DD_2$. The limiting Hodge operators which are Jacobi matrices in finite dimensions
have the spectrum on the Julia set of quadratic maps. See also \cite{BGH,BGH}.

\begin{figure}
\scalebox{0.05}{\includegraphics{figures/triangle.pdf}}
\scalebox{0.05}{\includegraphics{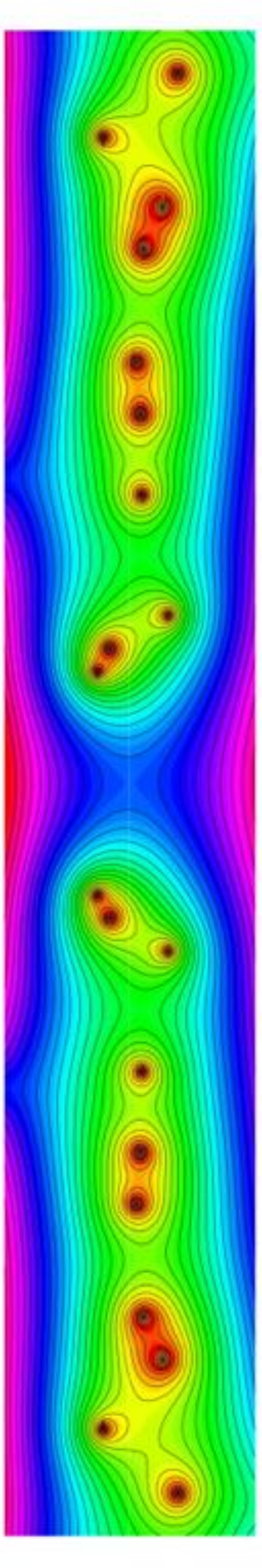}}
\scalebox{0.05}{\includegraphics{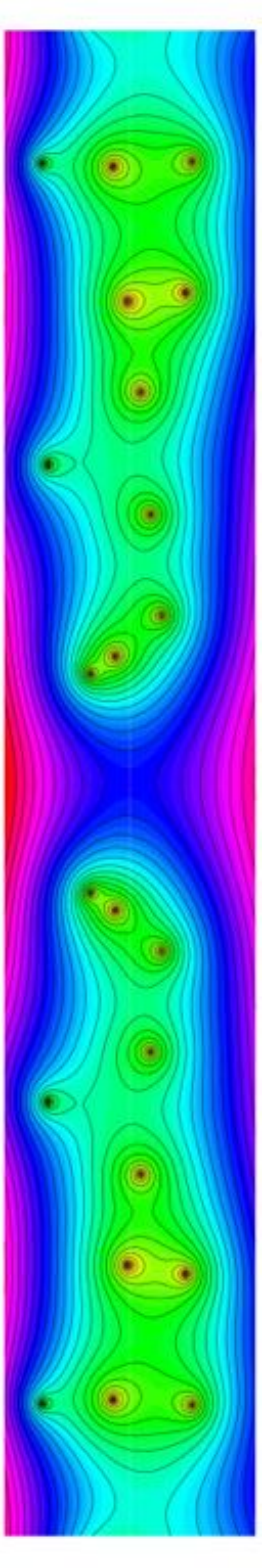}}
\scalebox{0.05}{\includegraphics{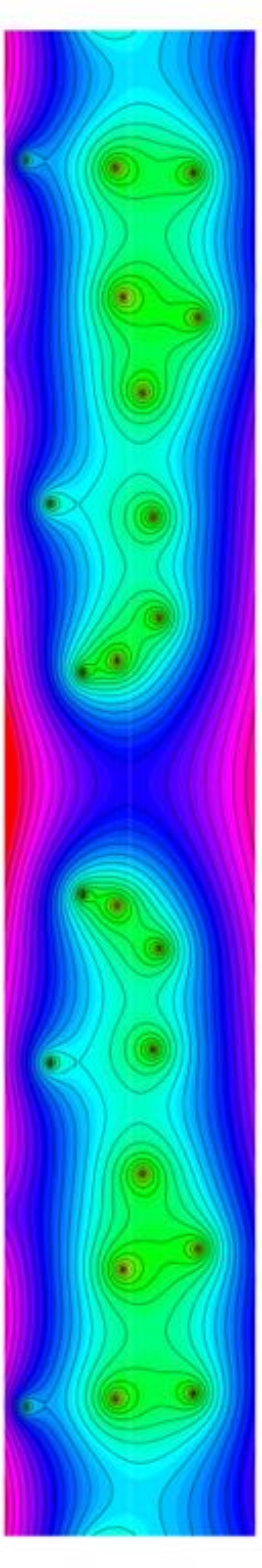}}
\caption{
The spectrum of Barycentric refinements $G_i$ of the triangle $G_0=K_3$. 
We don't know the density of states of $L(G_i)$ in the Barycentric limit yet. We also
don't know the limiting distribution of the zeros of $\zeta(G_i)$ but expect it to 
converge to a limiting measure in the complex plane. 
}
\end{figure}

\paragraph{}
Both the Hodge Laplacian $H$ as well as the connection
Laplacian $L$ have a Barycentric limit, an almost periodic operator on the 
group of dyadic integers. In both cases, the limiting operator is
the same as long as we start with a $1$-dimensional complex. The limiting 
connection operator is invertible. It has a ``mass gap". The dyadic group of integers
is compact but discrete. It is the dyadic analogue of the circle but unlike the 
circle $\TT$ which has a continuum of translations, there is a smallest translation 
on $\DD_2$. The limiting Riemann zeta function behaves differently than the Riemann zeta function
of the circle and appears more approachable. In the Hodge case, we have only a vague intuition 
about the limiting distribution of roots. 

\begin{figure}
\scalebox{0.03}{\includegraphics{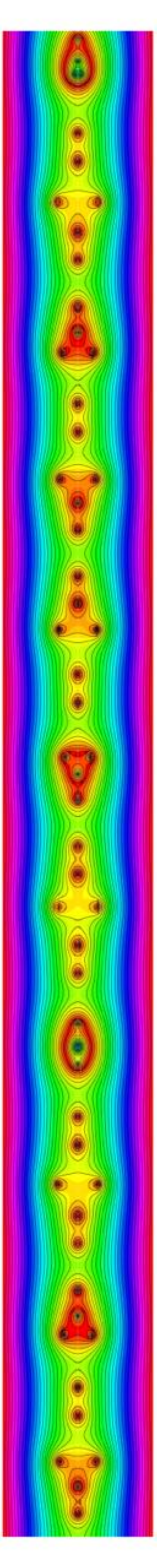}}
\scalebox{0.03}{\includegraphics{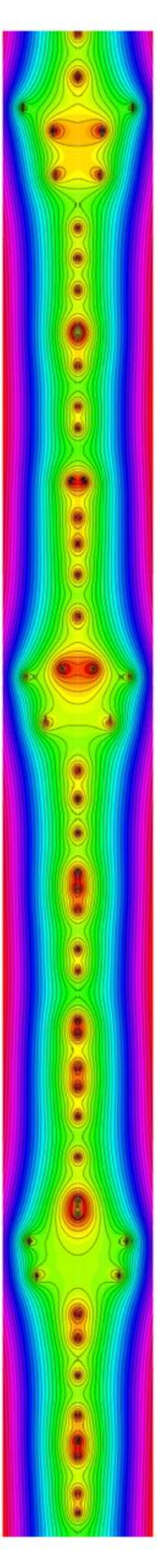}}
\scalebox{0.03}{\includegraphics{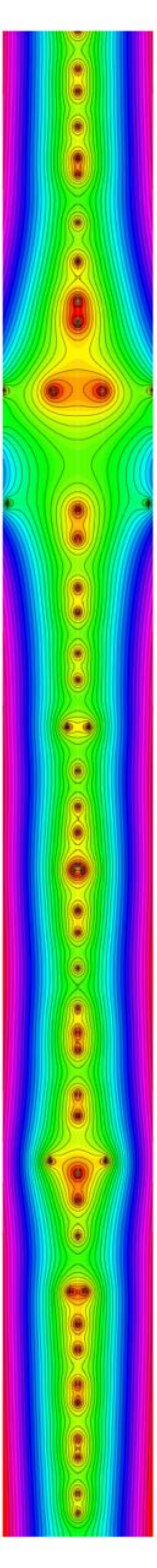}}
\scalebox{0.03}{\includegraphics{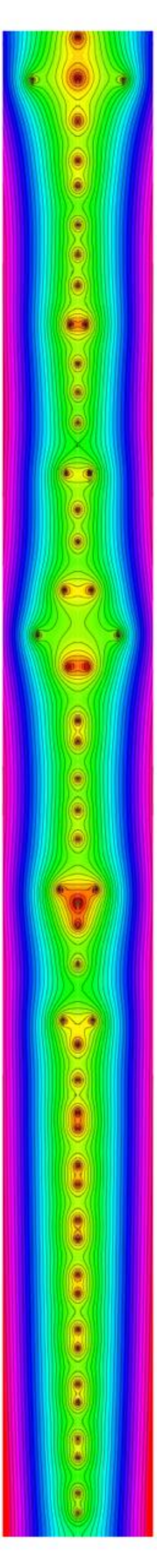}}
\scalebox{0.03}{\includegraphics{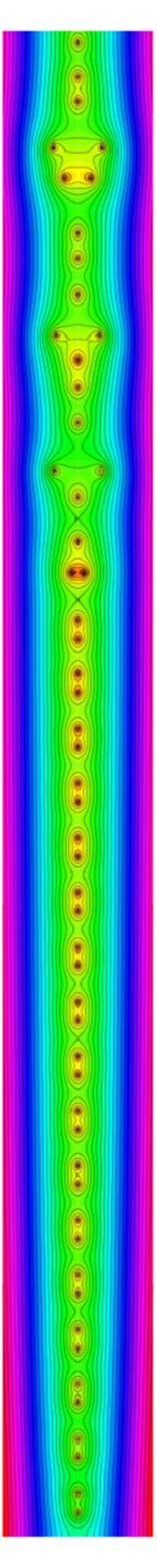}}
\scalebox{0.03}{\includegraphics{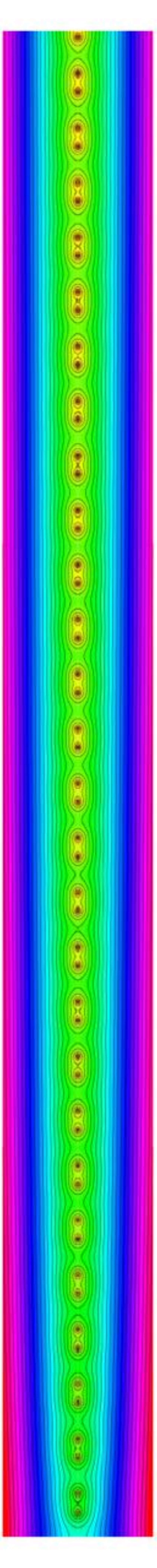}}
\caption{
We see the level curves of $|\zeta_G(s)|$ for the first
few Barycentric refinements of a discrete circle
$G=C_4,C_8,C_{16},C_{32},C_{64},C_{128}$. We see in each
of the 6 cases the rectangle $[-3,3] \times [0,60]$.
The roots converge more and more to the imaginary axes,
the line of symmetry of the functional equation.
}
\end{figure}

\section{Questions}

\begin{itemize}

\item We still need a reference or argument which assures that the entire function
$$ \zeta(it)=_4F_3\left(\frac{1}{4},\frac{3}{4},-it,it;\frac{1}{2},\frac{1}{2},1;-4\right) $$
has all roots on the real axes. We tried to rewrite this function in various ways, but could
place it in a class of functions which have real roots. Finite Pochhammer sums 
like $\sum_k a_k (s)_n (-s)_n$ or cos-transforms of measures on the real line can in general 
have non-real roots too. We tried contour integration and Sturmian criteria so far. 
Why was it possible in a similar Hodge case that roots are absent for ${\rm Re}(s)<0$? 
The reason was that we had Riemann sums of an explicit function
$g(s) = \int_0^1 4 \sin^s(\pi x) \; dx$for which 
$$ g(s) = \frac{4 \Gamma(\frac{1+s}{2})}{\sqrt{\pi} \Gamma(\frac{s}{2}+1)} $$
reveals that there are no roots because the Gamma function has no roots. 
% Integrate[ 4 Sin[Pi x]^s,{x,0,1}]

\item We would like to have explicit expressions for the roots. We would also like to know
in the case of each $G=C_n$, at which place the first pair of roots away on the imaginary axes appears. We have not
investigated whether there are any circular graphs $C_n$ for which all roots are on the imaginary axes. 
A ``lock-in" situation, meaning the existence of an $n_0$ such that for all $n>n_0$, all roots are on the 
imaginary axes would be a big surprise but it is not excluded at this point. 
An other open question is whether there are higher dimensional
complexes for which all roots of the zeta function are on the imaginary axes. We can give products
of simplicial complexes for which that happens but these are not simplicial complexes any more. 

\item In which cases does the spectrum of $L^2$ determine the spectrum of $L$? We know
that in general, it does not. How large can the 
multiplicity of eigenvalues become? We see that is rare that negative
and positive eigenvalues cancel each other when adding the spectrum. We looked
experimentally for cases where $\lambda_k+\lambda_l=0$ and found only very few. So far,
an example of a simplicial complex $G$ with $f$-vector $(14, 36, 27, 9)$ was found, where two 
eigenvalues $\pm (1 + \sqrt{5})/2$ appear. 
The case of non-simple spectrum $\lambda_k=\lambda_l$ happened so far only for 
$\lambda=1$ or $\lambda=-1$, which for $1$-dimensional complexes linked to the $0$-dimensional
and $1$-dimensional cohomology. 

% COMMENT 16

\item In general, given any simplicial complex $G$, the complex $G$ generates a monoid 
in the strong ring for which the union of the spectrum is a multiplicative subgroup of $R$. 
As a finitely presented Abelian group, it has a Pr\"ufer rank. What is this rank?
Is this rank invariant under Barycentric refinements? 

\item There is more about meaning of the eigenvalues $1$ or $-1$. 
Their multiplicity are important combinatorial invariants:
if $G$ is a Barycentric refinement of a $1$-dimensional complex, then 
the eigenvalues $1$ belong to the zero'th cohomology $H^0(G)$, the eigenvalues $-1$ belong to the
first cohomology $H^1(G)$. The eigenfunctions to the eigenvalue $1$ are supported on the zero-dimensional
part, the eigenfunctions to the eigenvalue $-1$ are supported on the $1$-dimensional part. 
We have seen earlier examples showing that one can not hear the cohomology. Yes, but 
these examples were not Barycentric refinements. Indeed, the eigenfunctions define a vertex
or edge coloring and require the graph to be bipartite. Every Barycentric refinement of a 
$1$-dimensional complex is trivially bipartite, the vertex dimension providing a 2-coloring.

\item For two or higher dimensional simplicial complexes, we have no idea yet
how to get the cohomology from the spectrum or zeta function or whether it is accessible
from the roots in the Barycentric limit. 
Exploring this numerically is harder as the number of simplices explodes
much faster than in the $1$-dimensional case. 

\item A matrix $L$ is called totally unimodular, if every 
invertible minor in $L$ is unimodular, meaning ${\rm det}(L_P) \in \{-1,0,1\}$.
An interesting question is in which cases the 
connection matrix $L$ of a $1$-dimensional complex is unimodular.
It is already not for circular graphs but appears so for linear graphs. 

\item It would be nice to have an analogue statement relating
the Riemann sum approximation $\zeta_n(s)/n-\zeta(s)$ asymptotics
with the existence of roots of $\zeta(s)$ at a point $s$.  

% COMMENT 17

\end{itemize}

\section{Mathematica Code}

\paragraph{}
The code can be grabbed by looking at the LaTeX source on the ArXiv.
Each of the blocks is self-contained. We first compute $\zeta(C_n)$
of a circular graph and compare with the explicit formulas for the eigenvalues. 
Then we compare the zeta function value with the limiting value.  

\begin{tiny}
\lstset{language=Mathematica} \lstset{frameround=fttt}
\begin{lstlisting}[frame=single]
(* A)   Explicit Connection spectrum in circular case               *)
Generate[A_]:=Delete[Union[Sort[Flatten[Map[Subsets,A],1]]],1];
M=26; A=Table[{k,Mod[k,M]+1},{k,M}]; G=Generate[A]; n=Length[G];
L=Table[If[DisjointQ[G[[k]],G[[l]]],0,1],{k,n},{l,n}];
EV=Sort[Eigenvalues[1.0*L]]; W=EV^2;Z[s_]:=Sum[W[[k]]^(-s/2),{k,n}]/n;
mu[k_]:=4 Sin[Pi k/M]^2; Q=Sort[Table[1.0 mu[k],{k,M}]];
f1[y_]:=(y + Sqrt[4 + y^2])/2; f2[y_]:=(y - Sqrt[4 + y^2])/2;
V=Sort[Flatten[{Map[f1,Q],Map[f2,Q]}]]; Print[V==EV]; 
X[t_]:=HypergeometricPFQ[{1/4,3/4,-t/2,t/2},{1/2,1/2,1},-4];
T=10*Random[]; Print[X[I T]-Z[I T]]   (* How close to limit? *)
\end{lstlisting}
\end{tiny}

Now, we first generate a random $1$-dimensional complex,
then check the functional equation.

\begin{tiny}
\lstset{language=Mathematica} \lstset{frameround=fttt}
\begin{lstlisting}[frame=single]
(* B)   Functional equation for one dimensional complexes           *)
Generate[A_]:=Delete[Union[Sort[Flatten[Map[Subsets,A],1]]],1];
sort[x_]:=Sort[{x[[1]],x[[2]]}]; Gr=RandomGraph[{9,20}];
A=Union[Map[sort,EdgeList[Gr]]]; G=Generate[A]; n=Length[G]; 
L=Table[If[DisjointQ[G[[k]],G[[l]]],0,1],{k,n},{l,n}];
V=Sort[Eigenvalues[1.0*L]]; 
Print[Total[Chop[Sort[V^2]]-Sort[1/V^2]]]
\end{lstlisting}
\end{tiny}

Here we collect various expressions for the
limiting Zeta function which were mentioned in the text. 

\begin{tiny}
\lstset{language=Mathematica} \lstset{frameround=fttt}
\begin{lstlisting}[frame=single]
(* C)  Rewriting the limiting Zeta Function                         *)
NI=NIntegrate; Clear[U,V,W,p,P,h,H,X,T]; 
U[t_]:=NI[Cos[t*Log[(2*Sin[Pi*x]^2+Sqrt[1+4*Sin[Pi*x]^4])]], {x,0,1}]; 
V[t_]:=NI[2*Cos[t*Log[2*u^2+Sqrt[1+4*u^4]]]/(Pi*Sqrt[1-u^2]),{u,0,1}]; 
W[t_]:=NI[2*Cos[t*Log[2*v+Sqrt[1+4*v^2]]]/(2Pi*Sqrt[v(1-v)]),{v,0,1}]; 
p[z_]:=(1+z)(1-z)(z^2-4z-1); a=1;b=2+Sqrt[5]; 
P[t_]:=Re[NI[(1+z^2)(z^(I*t)+z^(-I*t))/(2Pi*z*Sqrt[p[z]]),  {z,a,b}]];
h[y_]:=Sqrt[Cosh[y] Coth[y]/(2-Sinh[y])]; 
H[t_]:=NI[Cos[t*y]*h[y],{y,0,Log[2+Sqrt[5]]}]/Pi;  (* Cos Transform *)
X[t_]:=Re[HypergeometricPFQ[{1/4,3/4,-I t/2,I t/2},{1/2,1/2,1},-4]];
T=10*Random[];  Print[{U[T],V[T],W[T],P[T],H[T],X[T]}];
\end{lstlisting}
\end{tiny}

For a general $G$ we compute $\chi(G)$ and ${\rm det}(L)$ of $L$ 
using the zeta function $\zeta(G)$.
Also illustrated is the energy theorem $\chi(G) = \sum_{x} \sum_y g(x,y)$
where $g=L^{-1}$. 

\begin{tiny}
\lstset{language=Mathematica} \lstset{frameround=fttt}
\begin{lstlisting}[frame=single]
(* D)  Euler Chi and Det from Zeta                                  *)
Generate[A_]:=Delete[Union[Sort[Flatten[Map[Subsets,A],1]]],1]
R[n_,m_]:=Module[{A={},X=Range[n],k},Do[k:=1+Random[Integer,n-1];
  A=Append[A,Union[RandomChoice[X,k]]],{m}];Generate[A]];
G=R[10,12];n=Length[G]; Dim=Map[Length,G]-1;
Fvector=Delete[BinCounts[Dim],1]; Vol=Total[Fvector];
L=Table[If[DisjointQ[G[[k]],G[[l]]],0,1],{k,n},{l,n}];
Pos=Length[Position[Sign[Eigenvalues[1.0*L]],1]];
Bos=Length[Position[Flatten[Map[OddQ,Map[Length,G]]],True]];
Fer=n-Bos; Euler=Bos-Fer; Fred=Det[1.0*L]; Fermi=(-1)^Fer;
Energy=Round[Total[Flatten[Inverse[1.0 L]]]];
Checks={Energy,Fred,Pos,Energy==Euler,Fred==Fermi,Pos==Bos,Vol==n}
EV=Sort[Eigenvalues[1.0*L]]; zeta[s_]:=Sum[EV[[k]]^(-s),{k,n}];
a=Chop[ I zeta'[0]/Pi ]; Print[{Chop[a],Fer}];
Print[{Euler,Chop[zeta[0]-2 I zeta'[0]/Pi]}];
\end{lstlisting}
\end{tiny}

To illustrate the proof of the proposition we compute 
the minors of the deformation $K(t)$, then 
the coefficients of the characteristic polynomial $p$. The polynomial $p$
is also computed directly. In both cases, the function $\delta_k(t)$ appearing
in the artillery proposition is given. 

\begin{tiny}
\lstset{language=Mathematica} \lstset{frameround=fttt}
\begin{lstlisting}[frame=single]
(* E) Illustrating the artillery Proposition                        *)
G={{1},{2},{3},{1,2},{2,3}}; n=Length[G];  
L=Table[If[DisjointQ[G[[k]],G[[l]]],0,1],{k,n},{l,n}]; K=L;
Do[If[(k==n||l==n)&&k+l!=2n,K[[k, l]]=t K[[k, l]]],{k,n},{l,n}];
m=Table[Flatten[Minors[K,k]],{k,0,n}];
q=Table[m[[k+1]].m[[k+1]],{k,0,n}]; Phi=Simplify[q-Reverse[q]];
p=-CoefficientList[CharacteristicPolynomial[K.K,x] /. x->-x,x]
phi=Simplify[p-Reverse[p]]; Print[Phi-phi]; Print[phi];
\end{lstlisting}
\end{tiny}

Finally, let us look at the Riemann sum approximation and especially the Friedli-Karlsson
function $h_n(s)$ \cite{FriedliKarlsson}. We check numerically $h_n(1-s)/h_n(s)$ goes to
one. Equivalent to the Riemann hypothesis is that  
the convergence holds for all $s$ in the critical strip.

\begin{tiny}
\lstset{language=Mathematica} \lstset{frameround=fttt}
\begin{lstlisting}[frame=single]
(* F)    Riemann Sum in Hodge case                                 *)

h[n_,s_]:=Module[{z,Z},
  z=2^(-s) Sum[1/Sin[Pi k/n]^s,{k,n-1}]; (* discrete zeta function *)
  Z=Gamma[1/2-s/2]/(2^s*Sqrt[Pi]*Gamma[1-s/2]); (* limiting zeta   *)
  (4 Pi)^(s/2) Gamma[s/2] n^(-s)*(z-n Z)];    (* asymptotic        *)
n=1000; s=Random[]+400 I Random[]; Print[Abs[h[n,1-s]/h[n,s]]];
\end{lstlisting}
\end{tiny}

In our case, with the connection Laplacian, the approximation is much
better. While before, one had to compare the value at $s$ with the value
at $1-s$, we have here total symmetry $s \to -s$ also in the finite
case. It appears as if $|\zeta_n/n-\zeta|$ is incredibly good.
The following experiment probes whether the error is of 
order $1/n^3$. 

\begin{tiny}
\lstset{language=Mathematica} \lstset{frameround=fttt}
\begin{lstlisting}[frame=single]
(* G)    Riemann Sum in connection case                           *)

H[n_,s_]:=Module[{m=2n,Z,z,q}, q[x_]:=Abs[x/2]^s; 
Z=HypergeometricPFQ[{1/4,3/4,-s/2,s/2},{1/2,1/2,1},-4];  (* Dyadic*)
z=Sum[y=4Sin[Pi k/n]^2;u=Sqrt[4+y^2];q[y+u]+q[y-u],{k,n}];
Abs[m^2(z-m Z)]]; s=Random[]+10*I Random[]; Print[H[10000,s]];
\end{lstlisting}
\end{tiny}

The asymptotic error as a function of of $n$ and $s$ still needs to be
investigated. Also interesting would be to know what happens
in the random case, in the large $n$ limit of some Erd\"os-R\'enyi
spaces $E(n,p)$, where the Whitney complex of a graph with $n$
vertices is taken in which each edge is turned on with probability
$p$. The pictures for small $n$ suggest that the zero-locus
measure could converge to a continuous measure for $n \to \infty$. 

\begin{figure}
\scalebox{0.04}{\includegraphics{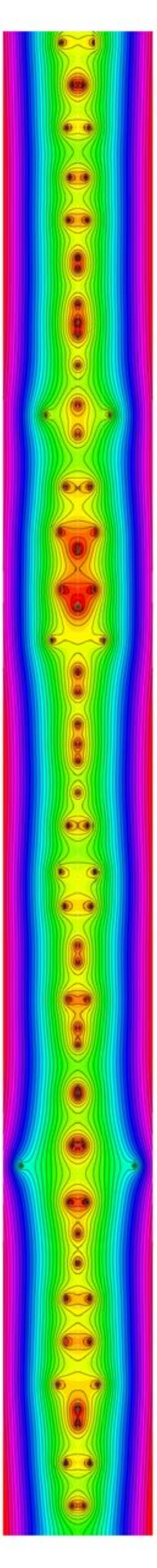}}
\scalebox{0.04}{\includegraphics{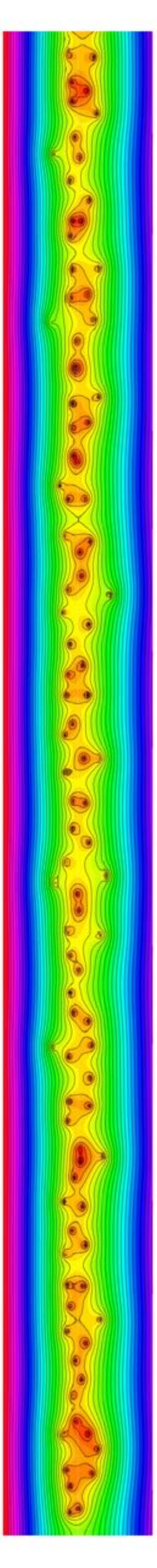}}
\scalebox{0.04}{\includegraphics{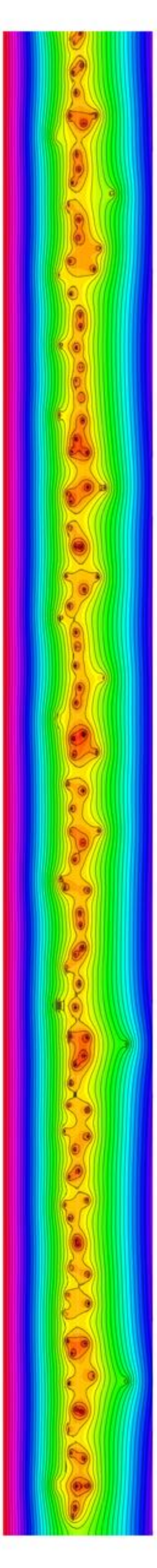}}
\scalebox{0.04}{\includegraphics{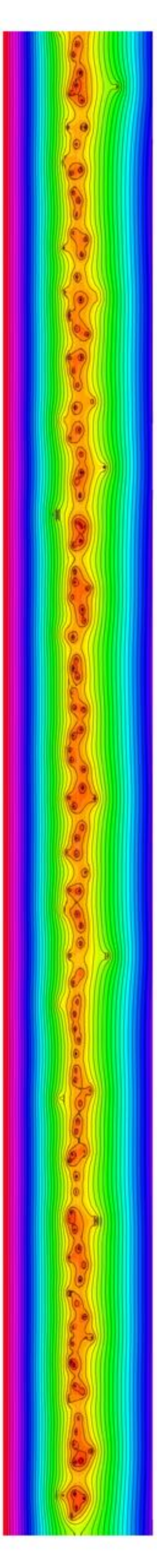}}
\scalebox{0.04}{\includegraphics{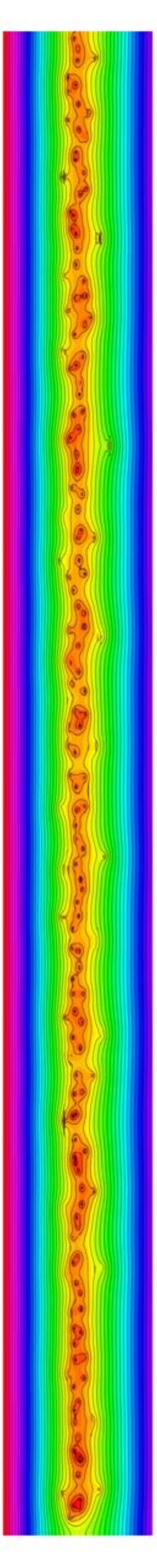}}
\scalebox{0.04}{\includegraphics{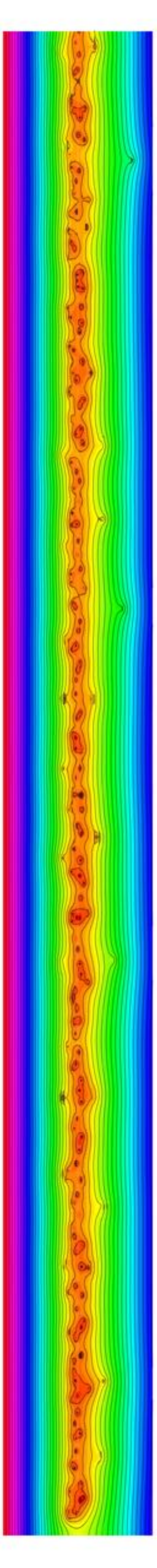}}
\caption{
The level curves of $|\zeta_{L^2}|$ for the Whitney complexes
of some Erd\"os-R\'enyi type graphs in $E(20,p)$, where the
edge probability is $p=0.1,0.2,\dots, 0.6$. 
In this case, the complexes had $70,108,169,288$, and $600$ 
simplices.
}
\end{figure}

\begin{figure}
\scalebox{0.04}{\includegraphics{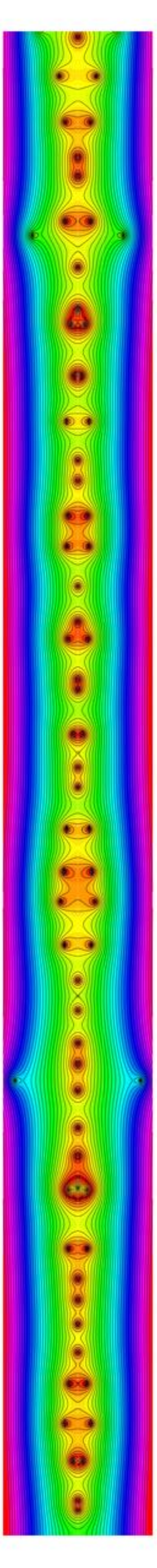}}
\scalebox{0.04}{\includegraphics{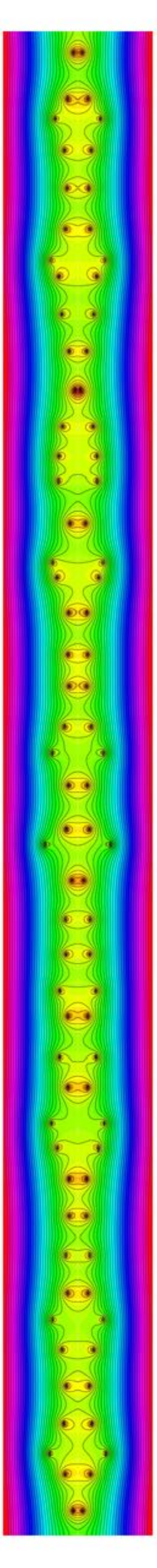}}
\scalebox{0.04}{\includegraphics{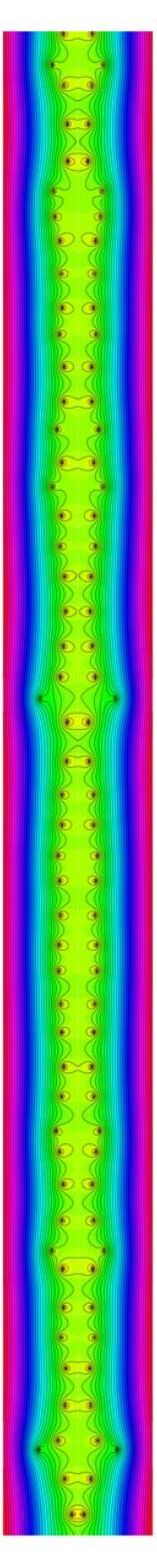}}
\scalebox{0.04}{\includegraphics{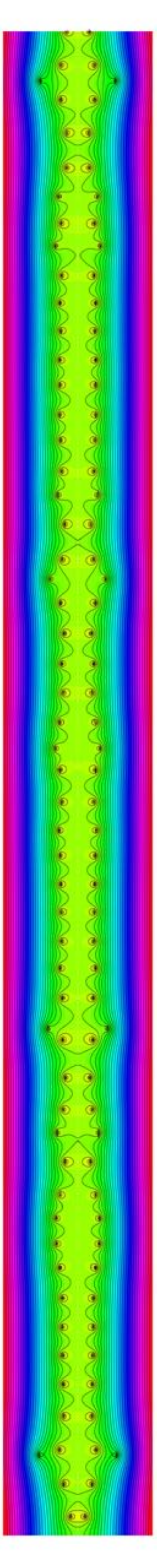}}
\scalebox{0.04}{\includegraphics{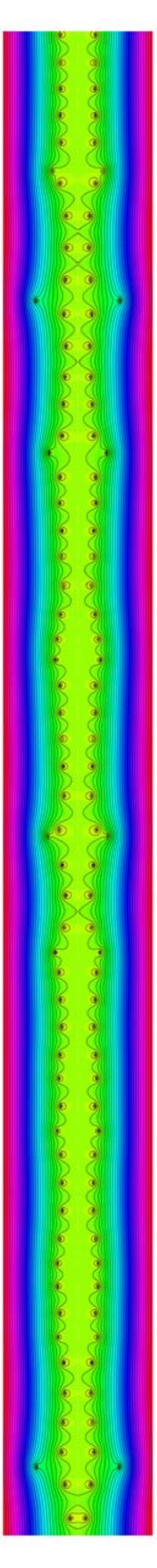}}
\scalebox{0.04}{\includegraphics{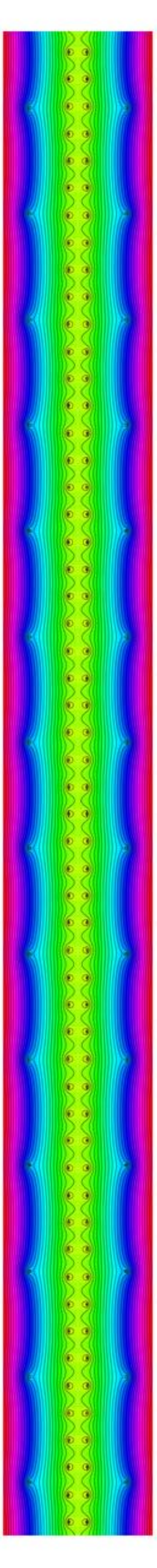}}
\caption{
The level curves of $|\zeta_{L^2}|$ for the 1D skeleton
complexes of graphs in $E(20,p)$ with
$p=0.1,0.2,\dots, 0.5$ and $p=1$.
In this case, the complex size is $|V|+|E|=20+190 p$. We see
the spectral symmetry.
% G=CompleteOneComplex[10]; L=ConnectionLaplacianComplex[G]; Eigenvalues[L.L]
}
\end{figure}

\vfill

\bibliographystyle{plain}

\end{document}